\documentclass[journal]{IEEEtran}
\usepackage{cite}
\usepackage{url}
\usepackage{amsmath}
\allowdisplaybreaks
\usepackage{amssymb}
\usepackage{enumitem}
\usepackage{bm}
\usepackage{xcolor}  
\usepackage{hyperref}   % 通常用于生成PDF链接，建议放在cleveref之前
\usepackage{cleveref}
\hypersetup{
	colorlinks=true,     % Set to true to color the link text instead of the box
	linkcolor=blue,      % Color for internal links (sections, figures, etc.)
	citecolor=blue,     % Color for citations (e.g., from bibtex)
	urlcolor=magenta     % Color for external URLs
}
\usepackage{amsmath, amssymb}
\usepackage{algorithm,algorithmic}
\usepackage{mathtools}
\usepackage[caption=false,font=footnotesize,subrefformat=parens,labelformat=parens]{subfig}
\usepackage{refcount}
\usepackage{color}\definecolor{RED}{rgb}{1,0,0}\definecolor{BLUE}{rgb}{0,0,1}
\newcommand{\mathletter}[1]{%
	\expandafter\newcommand\csname b#1\endcsname{\mathbb #1}
	\expandafter\newcommand\csname c#1\endcsname{\mathcal #1}
	\expandafter\newcommand\csname f#1\endcsname{\mathfrak #1}
	\expandafter\newcommand\csname til#1\endcsname{\widetilde #1}
	\expandafter\newcommand\csname ha#1\endcsname{\widehat #1}
	\expandafter\newcommand\csname bf#1\endcsname{\bf #1}
	\expandafter\newcommand\csname s#1\endcsname{\mathsf #1}
}%
\def\mathletters#1{\mathlettersB #1,,}
\def\mathlettersB#1,{\ifx,#1,\else\mathletter #1\expandafter\mathlettersB\fi}
\mathletters{A,B,C,D,E,F,G,H,I,J,K,L,M,N,O,P,Q,R,S,T,U,V,W,X,Y,Z}

\newcommand{\mathletterl}[1]{%
	\expandafter\providecommand\csname v#1\endcsname{\vec{#1}}
}%
\def\mathlettersl#1{\mathlettersC #1,,}
\def\mathlettersC#1,{\ifx,#1,\else\mathletterl #1\expandafter\mathlettersC\fi}
\mathlettersl{a,b,c,d,e,f,g,h,i,j,k,l,m,n,o,p,q,r,s,t,u,v,w,x,y,z}

\renewcommand{\vec}[1]{\mathbf{#1}}

\newtheorem{theo}{Theorem}
\newtheorem{lemma}{Lemma}
\newtheorem{assum}{Assumption}

\newtheorem{remark}{Remark}

\begin{document}

\title{Adaptive Polyak Stepsize with Level-value Adjustment for Distributed Optimization}
\author{Chen Ouyang,  Yongyang Xiong, Jinming Xu, Keyou You, \IEEEmembership{Senior Member, IEEE}, and Yang Shi, \IEEEmembership{Fellow, IEEE}% <-this % stops a space
\thanks{This work was supported in part by the National Natural Science Foundation
	of China (62203254). ({\em Corresponding author: Yongyang Xiong})}%
\thanks{C. Ouyang and Y. Xiong are with the School of Intelligent Systems Engineering, Sun Yat-Sen University, Shenzhen 518107, P.R China. E-mail: \texttt{ouych26@mail2.sysu.edu.cn}; \texttt{xiongyy25@mail.sysu.edu.cn}.}
\thanks{J. Xu is with the State Key Laboratory of Industrial Control Technology and the College of Control Science and Engineering, Zhejiang University, Hangzhou 310027, China. E-mail: \texttt{jimmyxu@zju.edu.cn}.}
\thanks{K. You is with the Department of Automation, Beijing National Research Center for Information Science and Technology, Tsinghua University, Beijing 100084, China. E-mail: \texttt{youky@tsinghua.edu.cn}.}
\thanks{Y. Shi is with the Department of Mechanical Engineering, University of Victoria, Victoria, BC V8W 2Y2, Canada. E-mail: \texttt{yshi@uvic.ca}.}
}

\maketitle

\IEEEpeerreviewmaketitle

\begin{abstract}
%	The selection of stepsizes is a critical challenge in distributed optimization over multi-agent networks. While adaptive methods like the Polyak stepsize are effective in centralized settings for their parameter-free nature, their direct application in distributed environments is non-trivial. 
Stepsize selection remains a critical challenge in the practical implementation of distributed optimization. Existing distributed algorithms  often rely on restrictive prior knowledge of global objective functions, such as Lipschitz constants. While centralized Polyak stepsizes have recently gained attention for their parameter-free adaptability and fast convergence. However, their extension to distributed settings is hindered by the requirement for local function values at the global optimum, which are typically unavailable to individual agents. To bridge this gap, we design a novel distributed adaptive Polyak stepsize algorithm with level-value adjustment (DPS-LA), where each agent only needs to solve a computationally efficient linear feasibility problem, thereby eliminating the dependency on global optimal values. Theoretical analysis proves that DPS-LA guarantees network consensus and achieves a linear speedup convergence rate of $\mathcal{O}(1/\sqrt{nT})$. Numerical results confirm the efficiency of the proposed algorithm.  
%This paper investigates distributed convex optimization problems with constraints over multi-agent networks. We propose a distributed gradient method based on the Polyak stepsize, which dynamically adjusts the stepsize by solving a linearly constrained optimization problem to estimate local function values at the global optimal solution progressively. However, since local gradient information cannot guarantee convergence to the global optimum, we further introduce a decaying variant of the distributed Polyak stepsize strategy. The proposed algorithm ensures exact convergence to the global optimum without relying on any prior problem specific parameters. 

%Theoretical analysis shows that the algorithm not only achieves consensus among all agents but also guarantees convergence to the globally optimal solution of the overall objective function. Moreover, we demonstrate that the function value based on the averaged estimates across the agents attains a convergence rate of \(\mathcal{O}(1/\sqrt{T})\), matching the optimal rate of centralized stochastic Polyak stepsize methods. This result confirms the effectiveness and theoretical competitiveness of the proposed distributed algorithm in terms of convergence performance.
\end{abstract}


\section{Introduction}
Distributed optimization has gained significant attention as a core computational framework for multi-agent systems, fueled by its extensive use in areas such as smart grids \cite{Li2022,LHQ2025}, multi-robot networks \cite{Xu2025IROS}. To improve scalability and communication efficiency, a variety of algorithms have been proposed, ranging from distributed gradient algorithms \cite{Wang2025} and primal-dual algorithms \cite{Xu2015}. While these algorithms differ in their design and communication schemes, their practical success remains deeply dependent on a single key element: the selection of an appropriate stepsize policy. 
%In recent years, numerous distributed optimization methods based on local gradient or subgradient information such as distributed gradient algorithms, distributed mirror descent algorithms, and distributed primal-dual algorithms have been developed to enhance scalability and communication efficiency \cite{7,8,9,10,11,12}.

%In the literature on distributed optimization, gradient-based methods have attracted significant attention due to their simplicity and scalability. Early foundational work by Nedić and Ozdaglar \cite{13} established a general framework for multi-agent optimization over networks, where they introduced distributed subgradient schemes and analyzed their convergence guarantees. This line of research was later extended to incorporate constraints on local decision variables, leading to the development of projected consensus algorithms \cite{14}. Other contributions include a accelerated schemes inspired by Nesterov’s method that enable faster convergence in distributed settings \cite{15}.
 
The selection of a suitable stepsize is a core element of optimization algorithms. An excessively large stepsize induces oscillation or even divergence, whereas an overly small one leads to prohibitively slow convergence and poor computational efficiency. Therefore, it is essential to choose a stepsize that allows the algorithm to not only converge but also reach the optimal solution, rather than merely settling within a neighborhood of the optimum. One of the primary challenges in distributed gradient-type algorithms arises from the fact that local gradients evaluated at the global optimum are generally non-zero. To achieve exact convergence to the optimal solution, diminishing stepsize strategies are common, provided they satisfy specific decay conditions \cite{Nedic2009}. However, these strategies suffer from slow convergence rates, limiting their practical applicability, especially in large-scale networked systems. In contrast, constant stepsizes, while facilitating faster initial convergence, typically guarantee convergence only to a neighborhood of the optimal solution, incurring a persistent steady-state error \cite{Nedic2009}. To address this fundamental trade-off and ensure exact convergence with constant stepsizes, recent research introduces various correction mechanisms and algorithmic modifications. Prominent among these are techniques such as gradient tracking \cite{Xu-2025,Huang}. Particularly, the practical deployment of these algorithms is often hindered by their heavy reliance on a priori knowledge, such as the connectivity of the network topology and Lipschitz constants. 

To alleviate this dependency and enhance algorithmic robustness, adaptive strategies receive significant attention, including distributed Barzilai--Borwein algorithms \cite{Gao2022,Yang2024} and automated stepsize schemes \cite{Chen2024,Chen2023a,Ren2021}. Despite these advances, most existing adaptive algorithms still necessitate global information, which is typically unavailable in distributed scenarios. Regarding parameter-free distributed optimization, while \cite{Ma2024} focuses on normalized gradient methods, the pioneering work \cite{Scutari2024} proposes a novel adaptive stepsize mechanism via local backtracking line search. This method operates without any global information and achieves a linear convergence rate. Recent studies \cite{Scutari2026, Scutari2026-2, Xu2025} further extend this framework to distributed composite optimization with convergence guarantee.

Among the various adaptive strategies, the Polyak stepsize is a particularly compelling candidate due to its exceptional performance in centralized optimization. Originally introduced by Polyak \cite{Polyak1987} for subgradient algorithms, this scheme undergoes significant theoretical expansion in recent years. In the finite-sum setting, the Stochastic Polyak Stepsize (SPS) converges to a neighborhood of the optimum \cite{Loizou2021}, while its decreasing variant DecSPS achieves an  \(\mathcal{O}(1/\sqrt{T})\) rate that matches the optimal complexity of SGD with diminishing stepsizes \cite{Orvieto2022}. Driven by its practical efficacy, extensive research focuses on extending the Polyak framework to broader contexts, including second-order variants \cite{Gower2022}, proximal formulations \cite{Ulbrich2023}, and momentum-accelerated algorithms \cite{Gu2025}. Despite their efficacy in centralized optimization, Polyak stepsize require the global optimum, which is generally unavailable in distributed settings.

In this work, we propose a distributed adaptive Polyak stepsize algorithm without knowing the global optimum. A novel level‑value adjustment technique is introduced to dynamically estimate local function values at the global information. Our approach requires each agent to solve only a lightweight linearly constrained problem per iteration. To ensure exact convergence in distributed settings, we further incorporate a decaying mechanism into the stepsize rule. Numerical experiments show that the proposed algorithm attains a faster convergence rate than the Distributed Gradient Descent (DGD) algorithm \cite{Nedic2010}. This performance improvement stems from their fundamentally different stepsize strategies, while allowing each agent to adaptively adjust its own optimal stepsize. The main contributions of this work are
summarized as follows.

(1) Algorithmically, we propose a novel distributed adaptive Polyak stepsize algorithm DPS-LA for distributed optimization. A simple example shows that directly applying the Polyak stepsize within the DGD framework leads to algorithmic divergence. By leveraging a level-value adjustment technique to estimate local function values through lightweight linear feasibility problems, the proposed algorithm eliminates the requirement for prior knowledge of the global optimum. 

(2) Theoretically, we show that DPS-LA achieves a sublinear convergence rate of $\mathcal{O}(1/\sqrt{nT})$, indicating linear speedup with respect to the number of agents $n$. Particularly, this implies that the total communication rounds for a given accuracy decrease proportionally with $n$. To the best of our knowledge, this represents the first theoretical guarantee for a distributed Polyak stepsize algorithm that operates without prior knowledge of the global optimal value. Numerical simulations further validate both the effectiveness of the algorithm.

The remainder of this paper is organized as follows. Section \ref{sec2} formulates the problem. The proposed DPS-LA algorithm is detailed in Section \ref{sec3}. Section \ref{sec4} establishes DPS-LA algorithm convergence results. Section \ref{sec5} provides numerical performance. Finally, \cref{sec6} concludes the paper.

\textit{Notation.} The vector of all ones in $\mathbb{R}^n$ is denoted by $\mathbf{1}_n$. For a matrix $W$, its $(i,j)$-th entry is denoted by $w_{ij}$. A non-negative square matrix $W \in \mathbb{R}^{n \times n}$ is called stochastic if its row sums are all equal to one (i.e., $W\mathbf{1}_n = \mathbf{1}_n$), and doubly stochastic if both its row and column sums are equal to one (i.e., $W\mathbf{1}_n = \mathbf{1}_n$ and $W^\text{T}\mathbf{1}_n = \mathbf{1}_n$). For a nonempty closed convex set $\mathcal{X} \subseteq \mathbb{R}^n$, the projection of a point $y$ onto $\mathcal{X}$ is defined as $P_\mathcal{X}(y) := \operatorname{argmin}_{x \in \mathcal{X}} \|y - x\|$. The distance from a point $x$ to the set $\mathcal{X}$ is $\text{dist}(x, \mathcal{X}) := \|x - P_\mathcal{X}(x)\|$, while the distance between two sets $\mathcal{X}$ and $\mathcal{Y}$ is $\text{dist}(\mathcal{X}, \mathcal{Y}) := \inf_{x \in \mathcal{X}, y \in \mathcal{Y}} \|x - y\|$. $\lfloor x \rfloor$ denote the largest integer not exceeding $x$. The notation  $y = \mathcal{O}(x)$ indicates that there exists a positive constant $M$ such that $y \leq Mx$.

\section{Preliminaries and Problem Formulation}\label{sec2}
In this section, we ﬁrst provide basic of the graph theory, and then formulate the  problem of interest.
\subsection{Graph Theory}
Consider an undirected graph $\mathcal{G}(\mathcal{V}, \mathcal{E})$, where $\mathcal{V} = \{1, \dots, n\}$ represents the set of agents, $\mathcal{E} \subseteq \mathcal{V} \times \mathcal{V}$ denotes the set of undirected links, and $(i, j) \in \mathcal{E}$ implies that agent $j$ can receive information from agent $i$. The neighbor set of agent $i$ is denoted by $\mathcal{N}_i = \{j : (j, i) \in \mathcal{E} \}$. Note that $(i, j) \in \mathcal{E}$ implies $(j, i) \in \mathcal{E}$ for undirected graph. The graph $\mathcal{G}(\mathcal{V}, \mathcal{E})$ is connected if there exists a path for any pair of two distinct agents. The interaction weight matrix associated with $\mathcal{G}(\mathcal{V}, \mathcal{E})$ is denoted by $W = [w_{ij}]_{n \times n}$ with $w_{ij} > 0$ if $(j, i) \in \mathcal{E}$ or $i=j$, and $w_{ij} = 0$, otherwise. Regarding the considered graph, we make the following assumption throughout this paper, which is commonly adopted in distributed optimization \cite{Nedic2018}.

\begin{assum}\label{a1}
	The undirected graph is connected and fixed.
\end{assum}

%	\begin{lemma}\label{lemma1}
%	A matrix $W \in \mathbb{R}^{n \times n}$ is irreducible if its associated undirected graph $\mathcal{G}$ is connected. If a non-negative square matrix $W \in \mathbb{R}^{n \times n}$ is irreducible and row stochastic, then the stochastic vector $\mathbf{a} = [a_1, a_2, \ldots, a_n]^\text{T} \in \mathbb{R}^n$ is a strictly positive vector satisfying
%	\[
%	\lim_{k \to \infty} W^k = \mathbf{1}\mathbf{a}^\text{T},
%	\]
%	where $\mathbf{1} = [1, 1, \ldots, 1]^\text{T} \in \mathbb{R}^n$ .
%\end{lemma}

\subsection{Problem Formulation}
Consider a network with $n$ agents, where each agent \(i\in \mathcal{V}\) privately holds a function \(f_i: \mathcal{X} \to \mathbb{R}\) defined over a common closed convex set \(\mathcal{X} \subseteq \mathbb{R}^m\). All agents collaboratively solve the following optimization problem
\begin{equation}\label{problem}
	\min_{x \in \mathcal{X}} \; 
	f(x) = \frac{1}{n} \sum_{i=1}^{n} f_i(x).
\end{equation}
%where each agent maintains a local variable \(x_i \in \mathbb{R}^m\), with \(x_{i,k}\) denoting its value at iteration \(k\). 
Let \(\mathcal{X}^{\star}\) be a non-empty set of optimal solutions. For \(x^{\star} \in \mathcal{X}^{\star}\), the optimal value is denoted by \(f^{\star} = f(x^{\star})\).  We make the following standard assumptions \cite{Ren2021, Nedic2010}.

\begin{assum}\label{a2}
	The constraint set $\mathcal{X}$ is compact, convex, and bounded.
\end{assum}
\begin{assum}\label{a3}
Each  $f_i$ is convex in its domain.
\end{assum}
Assumption \ref{a2} invokes the uniform boundedness of gradients on bounded sets \cite{Bertsekas2003}. 

		\section{Algorithm Development}\label{sec3}
In this section, we begin by introducing the standard Polyak stepsize and illustrate its limitations in distributed settings through a straightforward example. We then propose a horizontal value adjustment algorithm that dynamically estimates the global optimum required for the Polyak stepsize. Finally, we propose the DPS-LA algorithm.
	
	\subsection{Gradient Descent with the Polyak Stepsize Algorithm}
 Starting from an arbitrary initial point \( x_0 \in \mathbb{R}^m \), the Polyak stepsize update rule is defined as
	\begin{equation}\label{polyak}
		x_{k+1} = x_k - \frac{f(x_k) - f^{\star}}{\|g_k\|^2} g_k,
	\end{equation}
	where $g_k$ denotes the gradient of $f$ at \( x_k \). If $ \|g_k\| = 0$, then \( x_{k+1} = x_k \). Intuitively, the core of this stepsize lies in its use of the function-value gap \( f(x_k) - f^\star \) as an adaptive scaling factor. This design makes the stepsize automatically adjust according to both the current progress (through \( f(x_k) - f^\star \)) and the local geometry (through \( \|g_k\| \)). However, the primary challenge of the Polyak stepsize due to the lack of
	information about the $f^{\star}$. 	Recently, numerous estimation strategies are proposed to estimate the optimal value $f^{\star}$ in \eqref{polyak}, e.g., \cite{Oberman2021,Rabbat2021,Yu2025,Kiwiel1999}.
	\subsection{Polyak Stepsize with the Level-value Adjustment Algorithm}
 Liu et al. \cite{Liu2025} proposed the Polyak Stepsize Violation Detector (PSVD) to address the principal challenge of the unknown optimal value $f^{\star}$ in implementing the Polyak stepsize. They introduced a level-value $\bar{f}_k$ to approximate $f^{\star}$ in \eqref{polyak}
\begin{equation}\label{4}
	x_{k+1} = P_{\mathcal{X}}(x_k - \gamma \frac{f(x_k) - \bar{f}_k}{\|g_k\|^2} g_k), \quad 0 < \gamma < \bar{\gamma} < 2.
\end{equation}
where $\gamma$ and $\bar{\gamma}$ are constants. The core of this approach formulating a linear feasibility problem over a sliding window of $\eta+1$ iterations. Assuming a constant level-value $\tilde{f}$ that underestimates $f^{\star}$, the PSVD problem comprises $\eta+1$ linear inequalities
\begin{align}\label{psvd-ieq}
\left\{
\begin{aligned}
	(g_{k(j)})^\text{T} x &\le (g_{k(j)})^\text{T} x_{k(j)} - \frac{1}{\bar{\gamma}} s_{k(j)} \|g_{k(j)}\|^2 \\
	(g_{k(j)+1})^\text{T} x &\le (g_{k(j)+1})^\text{T} x_{k(j)+1} - \frac{1}{\bar{\gamma}} s_{k(j)+1} \|g_{k(j)+1}\|^2 \\
	&\vdots \\
	(g_{k(j)+\eta})^\text{T} x &\le (g_{k(j)+\eta})^\text{T} x_{k(j)+\eta} - \frac{1}{\bar{\gamma}} s_{k(j)+\eta} \|g_{k(j)+\eta}\|^2.
\end{aligned}
\right.
\end{align}
Geometrically, a feasible point $x$ must reside within the intersection of these half-spaces. Infeasibility of the PSVD problem indicates that the current level estimate is inconsistent with the observed optimization path, prompting an update to a tighter estimate $\bar{f}'$ according to
\begin{equation}\label{LV updata}
	\bar{f}' = \frac{\gamma}{\bar{\gamma}} \bar{f}_{k(j)+\eta} + \left(1 - \frac{\gamma}{\bar{\gamma}}\right) \min_{k \in [k(j), k(j)+1, \ldots, k(j)+\eta]} f(x_k),
\end{equation}
which ensures the new estimate is strictly bounded between the previous level and the optimum
\begin{equation*}
	\bar{f}_{k(j)+\eta} < \bar{f}' < f^{\star}.
\end{equation*}
The update rule \eqref{LV updata} offers a clear interpretation. Particularly, the level-value $\bar{f}'$ is a convex combination of the previous level-value and the minimum function value observed within the time window. As the gradient algorithm drives the iterates' function values $f(x^k)$ towards $f^{\star}$, this convex combination ensures that the level-value $\bar{f}'$  progressively tightened. As a result, $\bar{f}'$ tends to the optimal value $f^{\star}$. The PSVD thus provides a self-correcting mechanism that dynamically adjusts the level-value, accelerating convergence without requiring prior knowledge of $f^{\star}$.

\begin{remark}
	The core mechanism of the PSVD scheme can be interpreted geometrically through the lens of half-spaces. At each iteration $k$, the inequality \eqref{psvd-ieq} is equivalent to requiring that the optimal solution $x^{\star}$ lies within the half-space $H_k$, defined as
$
	H_k := \left\{ x \in \mathbb{R}^m : \frac{(-g_k)^\text{T} (x - x_k)}{\|g_k\|} \geq \frac{s_k \|g_k\|}{\bar{\gamma}} \right\}.
$
	This half-space represents all points sufficiently aligned with the descent direction, given the computed stepsize. The PSVD problem checks the consistency of a sequence of such half-spaces over multiple iterations. If the intersection $\bigcap_{i=k}^{k+T} H_i$ is non-empty, the half-spaces are mutually consistent, implying the stepsize (and thus the level-value) were appropriate. Conversely, infeasibility of this intersection provides a rigorous certificate that at least one half-space excludes $x^{\star}$, indicating a violation of the theoretical stepsize bound. This triggers an adjustment of the level-value, ensuring progressive refinement toward the true optimal value $f^{\star}$. 
\end{remark}

	\subsection{DPS-LA  Algorithm}
A critical observation is that the direct application of the Polyak stepsize to distributed gradient descent algorithm (DGD) results in algorithmic divergence. To illustrate this, we consider an undirected connected network comprising three agents in a triangular topology, with the following convex objective functions
\begin{align*}
	f_1(x) &= 2x_1^2 + 3x_2^2 + x_1x_2 - 4x_1 - 2x_2 \\
	f_2(x) &= x_1^2 + 4x_2^2 - 2x_1x_2 + 3x_1 - x_2 \\
	f_3(x) &= 3x_1^2 + 2x_2^2 + x_1 - 3x_2 + 2,
\end{align*}
where $x=[x_1,x_2]^\text{T}\in\mathbb R^2$ and $\mathcal X=\{x\in\mathbb R^2:\|x\|\le 4\}$.  
Agents communicate via a doubly-stochastic weight matrix $W$ constructed using the Metropolis rule \cite{MR}. Specifically, for the triangle graph, each edge is assigned a weight of $w_{ij}=1/3$. Figure~\ref{fig1} shows that the direct application of the Polyak stepsize fails to achieve consensus. This divergence originates from the structural tension inherent in distributed optimization, where the global objective in \eqref{problem} typically conflicts with the localized minima of individual functions. Within the DGD framework, agents navigate a dual mandate: performing local gradient descent while simultaneously synchronizing states via neighbor-based consensus. Consequently, a naive application of the Polyak stepsize \eqref{polyak} introduces a critical misalignment, as local function-value gaps fail to accurately reflect the global progress required for consensus. This mismatch ultimately triggers inconsistent update trajectories and leads to oscillatory behavior across the network.
\begin{figure}[t]
	\centering
	
	% 第一张图
	\subfloat{%
		\includegraphics[width=0.7\linewidth]{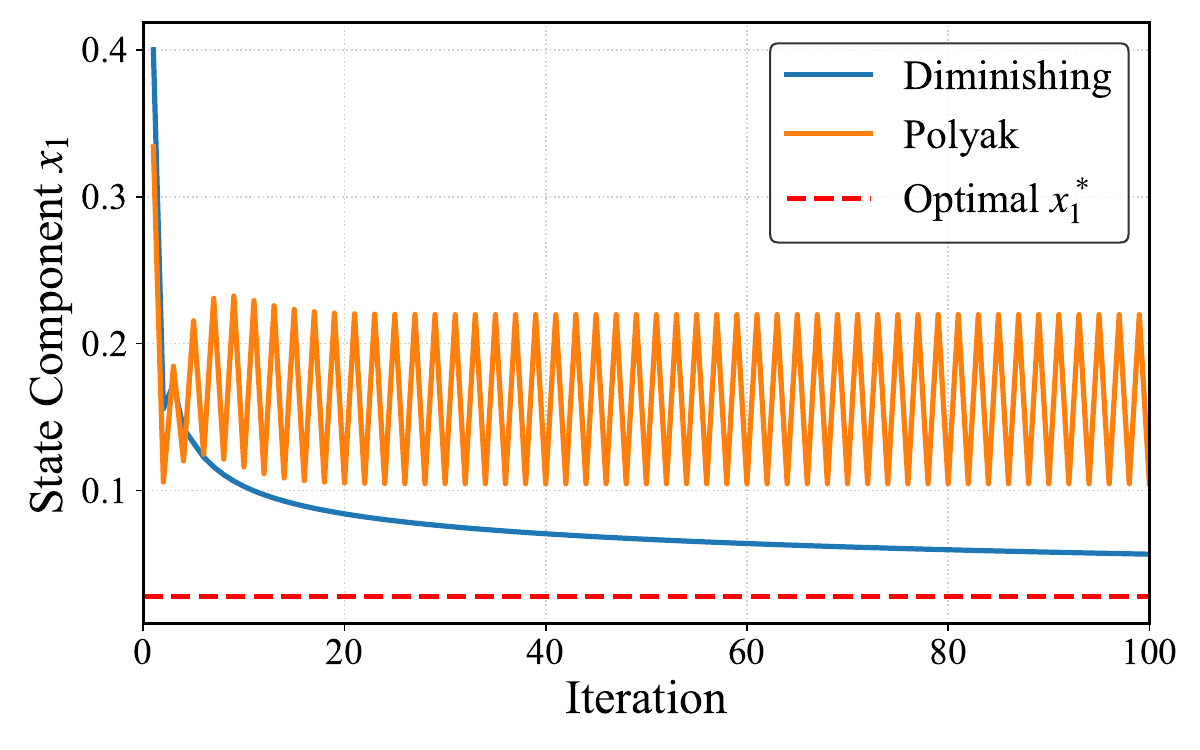}} \\
	% 两图之间的垂直间距
	\vspace{-1em} 
	% 第二张图
	\subfloat{%
		\includegraphics[width=0.7\linewidth]{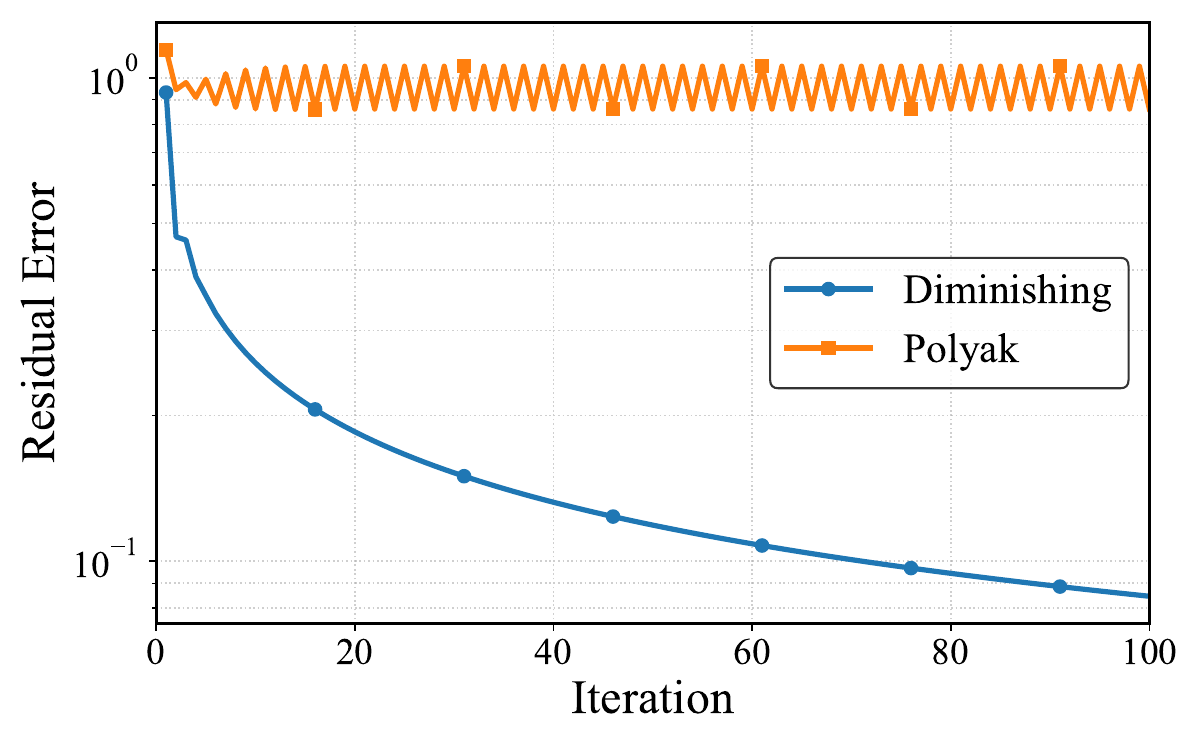}}
	\caption{In contrast to the diminishing stepsize, the direct integration of the Polyak stepsize into the distributed gradient scheme triggers algorithmic instability and eventual divergence.}
	\label{fig1}
\end{figure}

To overcome the aforementioned phenomena and obtain a suitable distributed Polyak stepsize algorithm, we first compute the individual stepsize $\beta_{i,k}$ on each agent $i\in \mathcal{V}$
\begin{align}\label{1}
	\beta_{i,k} =\gamma\frac{f_i(z_{i,k}) - {f}^{\star}_i}{\|\nabla f_i^k\|^2},
\end{align}
where $x_{j,k}$ represents the state of agent $j$ at iteration $k$, serving as a local estimate of the optimal solution. The variable $z_{i,k} = \sum_{j=1}^n w_{ij} x_{j,k}$ denotes the aggregated state obtained by fusing information from neighbors. Additionally, $f_i^{\star} := f_i(x^{\star})$ is the local function value at the global optimum, and $\nabla f_i^k$ is the gradient of $f_i$ evaluated at $z_{i,k}$. It is worth noting that due to the iterative process being
\begin{align*}
	 	x_{i,k+1} = P_{\mathcal{X}} \left( z_{i,k} - \beta_{i,k} \nabla f_i(z_{i,k}) \right),
\end{align*}
therefore, \(z_{i,k}\) will gradually converge towards the globally optimal point. Correspondingly, the value of \(f_i^{\star}\) in \eqref{1} should represent the local function's value at the globally optimal point. However, as with the standard Polyak algorithm, this value is often unknowable. Hence, we further develop a distributed level-value adjustment algorithm, wherein each agent maintains a linear inequality problem within a time window for $t = 0,1,\ldots,\eta$,
	\begin{align}\label{12}
		\left( \nabla f_i^{k_i(j)+t} \right)^{\text{T}} (x- z_{i,k_i(j)+t}) 
\leq - \frac{1}{\bar{\gamma}} \beta_{i,k_i(j)+t} \| \nabla f_i^{k_i(j)+t} \|^2,
\end{align}
When local problem \eqref{12} becomes infeasible, it indicates that agent $i$'s current level-value estimate $\bar{f}_i^{k_i(j)+\eta}$ is inconsistent with its optimization trajectory. The level-value is then updated to a tighter estimate
\begin{equation}\label{level}
	\bar{f}_i' = \frac{\gamma}{\bar{\gamma}} \bar{f}_i^{k_i(j)+\eta} + \left(1 - \frac{\gamma}{\bar{\gamma}}\right) \min_{k \in [k_i(j), \ldots, k_i(j)+\eta]} f_i(z_{i,k}).
\end{equation}
This update rule computes $\bar{f}_i'$ as a convex combination of the previous level-value and the minimum function value observed within the local time window, ensuring monotonic tightening towards $f_i^{*}$.

\begin{algorithm}[t]
	\caption{DPS-LA}
	\label{algorithm1}
	\begin{algorithmic}[1]
		\REQUIRE Given initial variables $x_{i,0} \in \mathcal{X}$; select parameters $\gamma$, $\bar{\gamma}$ such that $0 < \gamma < \bar{\gamma} < 2$; non-decreasing positive sequence $\{c_k\}_{k=0}^\infty$; $\alpha_{i,-1} = \alpha_0 > 0$ and level-value $\bar{f}_i^0 < f_i^{\star}$, $\forall i \in \mathcal{V}$; the weight matrix $W \in \mathbb{R}^{n \times n}$; and the maximum iteration number $K$.
		
		\FOR{$k = 0, 1, \ldots, K-1$}
		\STATE \textbf{Consensus step:} 
		$z_{i,k} = \sum_{j=1}^n w_{ij} x_{j,k}$.
		\STATE \textbf{Stepsize computation:}
		\begin{align*}
			\beta_{i,k} &= \gamma \frac{f_i(z_{i,k}) - \bar{f}^k_i}{\|\nabla f_i(z_{i,k})\|^2}, \\
			\alpha_{i,k} &= \frac{1}{c_k} \text{min}\left\{ \text{max}\left\{\beta_{i,k}, \frac{c_0\alpha_0}{2}\right\}, c_{k-1} \alpha_{i,k-1} \right\}.
		\end{align*}
		\STATE \textbf{State update:}
		$x_{i,k+1} = P_{\mathcal{X}}\left( z_{i,k} - \alpha_{i,k} \nabla f_i(z_{i,k}) \right)$.
		\STATE \textbf{Feasibility check and level-value update:}
		\IF{problem \eqref{12} with added constraint $(\nabla f_i(z_{i,k}))^{\text{T}} x \leq (\nabla f_i(z_{i,k}))^{\text{T}} z_{i,k} - \frac{1}{\bar{\gamma}} \beta_{i,k} \|\nabla f_i(z_{i,k})\|^2$ is infeasible}
		\STATE Update $\bar{f}^{k+1}_i$ using rule \eqref{level};
		\STATE Remove all previously added inequalities.
		\ELSE
		\STATE $\bar{f}^{k+1}_i = \bar{f}^k_i$.
		\ENDIF
		\ENDFOR
	\end{algorithmic}
\end{algorithm}

To guarantee exact convergence, we apply a decaying mechanism
\begin{equation*}
	\alpha_{i,k} = \frac{1}{c_k} \min \left\{  \max\{\beta_{i,k}, \frac{c_0\alpha_0}{2}\}, c_{k-1} \alpha_{i,k-1} \right\},
\end{equation*}
where $c_k$ is a non-decreasing positive sequence, $\alpha_0$ is a initial stepsize. The $\max$ operator ensures a controllable lower bound, while the $\min$ operator guarantees the overall decay of the stepsize. This scheme requires only the solution of simple linearly feasibility problems at each agent, eliminating the need for manual stepsize tuning. Accordingly, we propose a fully automated stepsize scheme for distributed optimization and learning, as outlined in Algorithm \ref{algorithm1}.

\begin{remark}
A pivotal design aspect in the stepsize computation (Step 2) is the aggregated variable $z_{i,k} = \sum_{j=1}^n w_{ij}x_{j,k}$. Unlike an agent's local state $x_{i,k}$ or a global state, $z_{i,k}$ constitutes a locally weighted average of neighboring states. This aggregation facilitates information diffusion, providing a local consensus reference that guides the agent's update. By querying $f_i$ and $\nabla f_i$ at $z_{i,k}$, agents collectively advance toward a common optimum. This design dictates the objective of the level-value adjustment. Since $z_{i,k}$ converges to the global optimum $x^{\star}$, the target for the Polyak stepsize is not the local minimum but $f_i(x^{\star})$. Our core innovation for approximating this target is a distributed level-value adjustment (Steps 5--6), which operates as an online cutting-plane algorithm. Each iteration, agent $i$ tests its level-value $\bar{f}_i^k$ by solving a local lightweight problem augmented with a new gradient-based inequality. Infeasibility indicates that $\bar{f}_i^k$ is an underestimate, prompting an update to a more conservative value. This mechanism creates an adaptive feedback loop, enabling agents to collaboratively learn $f_i(x^{\star})$.
\end{remark}

%\begin{figure}
%	\centering
%	\includegraphics[width=1\linewidth]{../1/test}
%
%\end{figure}
	\section{Convergence Analysis}\label{sec4}
In this section, we ﬁrst establish the permissible bounds for the stepsize. Subsequently, we prove that the level-value associated with each agent asymptotically approaches the global optimum, and that the agents achieve optimal consensus. Finally, the sublinear convergence rate of DPS-LA, which achieves a linear speedup, is rigorously established.

\begin{lemma}\label{bound}
	Consider a non-decreasing sequence of positive real numbers $\{c_k\}$. Then, for all $i \in \mathcal{V}$ and $k \geq 1$, the stepsize $\alpha_{i,k}$ is bounded and satisfies the monotonicity property $\alpha_{i,k} \leq \alpha_{i,k-1}$.
\end{lemma}

\begin{IEEEproof}
See Appendix \ref{L2proof}.
\end{IEEEproof}

\begin{lemma}\label{psvd}
	Given a fixed constant $\eta > 0$, the sequence $\{z_{i,k}\}_{k = k_i(j)}^{k_i(j)+\eta}$ is generated via gradients $\{\nabla f_i(z_{i,k})\}$ and parameters $\{\beta_{i,k}\}$ under constant level-values satisfying $\bar{f}_i^k < f_i^{\star}$. The PSVD problem for a decision vector \( x \in \mathbb{R}^m \) is defined by the following system of inequalities for $t = 0, 1, \ldots, \eta$:
	\begin{align*}
		&\left( \nabla f_i^{k_i(j)+t} \right)^{\text{T}} (x- z_{i,k_i(j)+t}) 
		&\leq - \frac{1}{\bar{\gamma}} \beta_{i,k_i(j)+t} \| \nabla f_i^{k_i(j)+t} \|^2.
	\end{align*}
	Infeasibility of this system implies the existence of an index \( k \in \{k_i(j), \ldots, k_i(j) + \eta\} \) such that
	\begin{equation*}
		\beta_{i,k} > \bar{\gamma} \frac{f \left( z_{i,k}  \right) - f_i^{\star}}{\| \nabla f_i^k\|^2}.
	\end{equation*}
	The corresponding updated level-value \( \bar{f}_i' \), defined as
	\begin{equation*}
		\bar{f}_i' = \frac{{\gamma}}{\bar{\gamma}} \bar{f}_i^{k_i(j)+\eta} + \left( 1 - \frac{{\gamma}}{\bar{\gamma}} \right) \underset{{k \in \{k_i(j), \ldots, k_i(j) + \eta\} }}{\text{min}} f_i \left( z_{i,k} \right), 
	\end{equation*}
	provides a strictly tighter lower bound satisfying $\bar{f}_i^{k_i(j)+\eta} < \bar{f}_i' < f_i^{\star}$.
\end{lemma}

\begin{IEEEproof}
	See Appendix \ref{L3proof}.
\end{IEEEproof}

	It is crucial to emphasize that while in the foundational work by Liu et al. \cite{Liu2025}, the level update mechanism is designed to converge to the agent's local optimum $f_i^{\star}$ in a single-agent context, its behavior is fundamentally altered in our distributed setting. The inclusion of a consensus protocol compels each agent's state $z_{i,k}$ to converge to the single global optimum $x^{\star}$, rather than exploring its own local function's landscape. The level update is driven by the history of evaluated function values. Since consensus forces $z_{i,k}$ converge to $ x^{\star}$, this history becomes dominated by values near $f_i(x^{\star})$. This directly implies that the level estimate $\bar{f}_i^k$ converges not to the local minimum, but to the agent's functional value at the global solution, $f_i(x^{\star})$. Thus, in the distributed framework, $\bar{f}_i^k$ effectively serves as an dynamic estimate of the agent's contribution to the global optimal value, a behavior confirmed by our numerical experiments.

Lemma~\ref{bound} provides a boundedness property of the step sizes, and Lemma~\ref{psvd} demonstrates that the estimated values progressively approach \(f_i^{\star}\). The subsequent result indicates that agents achieve consensus, where the consensus error is bounded by the gap between the level-value and the optimal value.

\begin{lemma}\label{lemma0}
	For a graph sequence satisfying Assumptions 1, the $c_k=\sqrt{k+1}$, and the optimization problem \eqref{problem} satisfying Assumptions 2 and 3, the agent's estimates $x_{i,k}$, $\forall i \in \mathcal{V}$, in the distributed gradient algorithm reach a consensus, i.e., $\lim_{k \to \infty} \| x_{i,k} - x_{j,k} \| = 0$, $\forall i, j \in \mathcal{V}$.
\end{lemma}
	\begin{IEEEproof}
The proof is similar to that of \cite{Nedic2010}.
\end{IEEEproof}
\begin{lemma}\label{lemma 7}
	Suppose that Assumptions 1--3 hold. Let $\{\bar{x}_k\}$, $\{z_{i,k}\}$, and $\{x_{i,k}\}$ be the sequences generated by Algorithm \ref{algorithm1}. Then, the average squared distance to $x^{\star}$ satisfies
	\begin{align}\label{lemma7}
		&\frac 1n \sum_{i=1}^n \| x_{i,k+1} - x^{\star} \|^2  - \frac 1n \sum_{i=1}^n \| x_{i,k} - x^{\star} \|^2 \notag \\
		&\leq \alpha_{\text{min}}\left(2 -\frac{\gamma}{c_k}\right) \left[ G \frac 1n \sum_{i=1}^n \|z_{i,k}-\bar{x}_k\| - \left(f(\bar{x}_k)-f(x^{\star})\right) \right] \notag \\
		&\quad + \frac{\gamma c_0 \alpha_0}{c_k^2} \frac 1n \sum_{i=1}^n \left(f_i(x^{\star})-\bar{f}_i^k\right),
	\end{align}
	where $\bar{x}_k \coloneqq \frac{1}{n} \sum_{i=1}^n x_{i,k}$ and $\alpha_{\text{min}} = \min\left\{ \frac{\gamma}{2c_k L_{\text{max}}}, \frac{c_0\alpha_0}{c_k} \right\}$.
\end{lemma}

\begin{IEEEproof}
	See Appendix \ref{L4proof}.
\end{IEEEproof}
\begin{lemma}\label{lemma8}
	Suppose Assumptions 1-3 hold, and with the parameter $c_k=\sqrt{k+1}$, it can be shown that there exists a subsequence $\{y(k_\nu)\}$ of the iterates $\{\bar{x}_k\}$ such that $\lim_{\nu \to \infty} f(y(k_\nu)) - f(x^{\star}) = 0$. For convenience in the subsequent analysis, we define $\sigma = f(x^{\star}) - \bar{f}^0$, where $\bar{f}^0$ is the initial common level-value for all agents.
\end{lemma}
\begin{IEEEproof}
	See Appendix \ref{L5proof}.
\end{IEEEproof}
With consensus guaranteed by Lemma \ref{lemma0} and the gap relationship established in Lemmas \ref{lemma 7} and \ref{lemma8}. We establish the convergence result of the proposed algorithm.
\begin{theo}[Convergence of DPS-LA]\label{theo1}
	Suppose that Assumptions 1--3 hold. The iterative sequences $\{x_{i,k}\}$ generated by Algorithm \ref{algorithm1} converge to a common optimal solution $x^{\star} \in \mathcal{X}^{\star}$ in the sense that 
	\begin{equation*}
		\lim_{k \to \infty} x_{i,k} = x^{\star}, \quad \forall i \in \mathcal{V}.
	\end{equation*}
	\end{theo}
\begin{IEEEproof}
	Define the $\epsilon$-sublevel set $\mathcal{D}_{\epsilon} = \{ y \in \mathcal{X} : f(y) - f(x^{\star}) \leq \epsilon \}$ and let $d(\epsilon) = \sup_{y \in \mathcal{D}_{\epsilon}} \text{dist}(y, \mathcal{X}^{\star})$. By convexity and compactness, $\lim_{\epsilon \to 0} d(\epsilon) = 0$. 

Similar the proof of Lemma \ref{lemma8},
when the $c_k=\sqrt{k+1}$, which implies decaying stepsizes for \(\alpha_{i,k}\), there exists a constant \(K'_1 \in \mathbb{N}^+\) such that for all \(k > K'_1\), the \(c_k\) satisfies
\[
\frac{1}{c_k}\leq \frac{\epsilon}{G \sqrt{\sigma \gamma c_0 \alpha_0}}.
\]

Additionally, there exists a constant \(K'_2 \in \mathbb{N}^+\) such that for all \(k > K'_2\), stepsize \(\alpha_{i,k}\)  satisfies
\[
\alpha_{i,k}\leq \frac{2\epsilon}{3},
\]
under Assumption 1, and for all \(k > K'_3\), by the consensus result, we get
\[
\frac{1}{n} \sum_{i=1}^{n} \|{\bar{x}_k - z_{i,k}}\| \leq \frac{\epsilon}{4G},
\]
for the difference between \(\|{\bar{x}_k - x^{\star}}\|\) and \(\|{x_{i,k} - x^{\star}}\|\), we have by the triangle inequality
\begin{align*}
	&\left| \|\bar{x}_k - x^{\star}\|^2 - \|x_{i,k} - x^{\star}\|^2 \right| \\
	&= \left| \|\bar{x}_k - x^{\star}\| - \|x_{i,k} - x^{\star}\| \right| \left( \|\bar{x}_k - x^{\star}\| + \|x_{i,k} - x^{\star}\| \right) \\
	&\leq \|\bar{x}_k - x_{i,k}\|\left( \|\bar{x}_k - x^{\star}\| + \|x_{i,k} - x^{\star}\| \right) ,
\end{align*}
then, from Lemma \ref{lemma0}, we know  \(\lim_{k\to \infty} \|{\bar{x}_k - x_{i,k}}\| =0\). Furthermore, since the underlying constraint set is bounded, it follows that both \(\|{\bar{x}_k - x^{\star}}\|\) and \(\|{x_{i,k} - x^{\star}}\|\) are bounded for all \(k\). This implies that if \(\bar{x}_k\) converges to \(x^{\star}\), then \(x_{i,k}\) also converges to \(x^{\star}\),
furthermore, for all \(k > K'_4\) (for some \(K_d \in \mathbb{N}^+\)), the average difference in squared distances can be bounded
\begin{align*}
	&|(  \|\frac{1}{n} \sum_{i=1}^n\bar{x}_k - \frac{1}{n} \sum_{i=1}^n x^{\star}\|^2 - \|\frac{1}{n} \sum_{i=1}^n x_{i,k} - \frac{1}{n} \sum_{i=1}^n x^{\star}\|^2) |\\
	&\leq \frac{1}{n} \sum_{i=1}^n |( \|\bar{x}_k - x^{\star}\|^2 - \|x_{i,k}- x^{\star}\|^2) |\\&\leq \frac{\epsilon^2}{4},
\end{align*}
from Lemma \ref{lemma8}, we analyze the iterate behavior in two cases:

\textit{Case 1:} $f(\bar{x}_k) < f(x^{\star}) + \epsilon$. In this case, $\bar{x}_k \in \mathcal{D}_\epsilon$. Using the property $\| a+b \|^2 \leq 2\| a \|^2 + 2\| b \|^2$ and the algorithm's update rule, we have
\begin{align*}
	&\frac{1}{n}\sum_{i=1}^{n}\|x_{i,k+1} - x^{*}\|^{2} \\
	&\leq	\frac{1}{n}\sum_{i=1}^{n}\| z_{i,k} - \alpha_{i,k} \nabla f_{i}^k  - x^{\star} \|^2 \\
	&\leq 2\left(\frac{1}{n}\sum_{i=1}^{n}\left\|z_{i,k}-  x^{\star}\right\|^{2} + \alpha_{i,k}^{2}\|\nabla f_i^k\|^2\right) \\
	&\leq 2\left(\frac{1}{n}\sum_{i=1}^{n}\|x_{i,k} -  x^{\star}\|^{2} +( \frac{c_0\alpha_0}{c_k}G)^2\right),
\end{align*}
when \(k > K_{4}^{\prime}\)
\begin{equation}\label{x bound t1}
	\begin{aligned}[b]
		&	\min_{x^{*} \in \mathcal{X}^{\star}} \frac{1}{n}\sum_{i=1}^{n}\|x_{i,k+1} -  x^{\star}\|^{2} \\
		&\leq \min_{x^{*} \in \mathcal{X}^{\star}} 2\left(\frac{1}{n}\sum_{i=1}^{n}\|x_{i,k} -  x^{\star}\|^{2} +( \frac{c_0\alpha_0}{c_k}G)^2\right) \\
		&\leq 2\min_{ x^{\star} \in \mathcal{X}^{\star}}\left(\|\bar{x}_k -  x^{\star}\|^{2} + \frac{\epsilon^{2}}{4} + (\frac{c_0\alpha_0}{c_k}G)^2\right) \\
		&\leq 2\left(d(\epsilon)^{2} + \frac{\epsilon^{2}}{4} +( \frac{c_0\alpha_0}{c_k}G)^2\right),
	\end{aligned}
\end{equation}
where the second inequality uses the bound on the difference between the average of squared norms and the squared norm of the average, and the third inequality follows from the definition of \(d(\epsilon)\).

\textit{Case 2:} $f(\bar{x}_k) \geq f(x^{\star}) + \epsilon$. Similar to the derivation in Lemma \ref{lemma8}, the evolution follows: when \(k > \max\{K_{1}^{\prime}, K_{3}^{\prime}\}\),
%\begin{align}
%		&\frac{1}{n}\sum_{i=1}^{n} \|x_{i,k+1} - x^{*}\|^{2} 
%-  \frac 1n \sum_{i=1}^n \| x_{i,k} - x^{\star} \|^2\notag \\ & \leq  G\alpha_{\text{min}}(2 -\frac{\gamma}{c_k})\frac 1n \sum_{i=1}^n  \|z_{i,k}-\bar{x}_k\|\notag\\
%	&\quad - G\alpha_{\text{min}}(2 -\frac{\gamma}{c_k})(f(\bar{x}_k)-f(x^{\star}))\notag \\
%	&\quad+ \frac{\gamma c_0 \alpha_0}{c_k^2} \frac 1n \sum_{i=1}^n   (f_i(x^{\star})-\bar{f}_i^k)\notag\\
%	&\leq \alpha_{\text{min}}(2 -\frac{\gamma}{c_k}) \frac{\epsilon}{4}- \alpha_{\text{min}}(2 -\frac{\gamma}{c_k}) \epsilon \notag \\
%		&\quad+ \frac{\gamma c_0 \alpha_0}{c_k^2} \frac 1n \sum_{i=1}^n   (f_i(x^{\star})-\bar{f}_i^k)),\notag
%\end{align}
%building upon Lemma \ref{lemma8}, we establish the  
%$
%\frac{1}{n}\sum_{i=1}^{n}(f_i(x^{\star})-\bar{f}_i^k) \leq f(x^{\star})-\bar{f}^0=\sigma,
%$
%%There exists a sufficiently large integer $K_5^{'} \in \mathbb{N}^+$ such that for all $k > K_5^{'}$, the avrage optimality gap satisfies:
%%\[
%%\frac{1}{n}\sum_{i=1}^{n}(f_i(x^{\star})-\bar{f}_i^k)-\sigma \leq \epsilon
%%\]
%%here, $K_0 = \max\{K_1^{'}, K_2^{'}, K_3^{'}\}$ encapsulates the conditions required for this bound to hold. 
%under this condition, the evolution of the average squared distance of the iterates $x_{i,k}$ to an optimal solution $x^{\star}$ can be bounded, we have
\begin{equation}\label{x bound t1,1}
	\begin{aligned}
		&\frac{1}{n}\sum_{i=1}^{n}\|x_{i,k+1} -  x^{\star}\|^{2} - \frac{1}{n}\sum_{i=1}^{n}\|x_{i,k}- x^{*}\|^{2}\notag\\
		& \leq \sigma \left( \frac{\gamma c_0 \alpha_0}{c_k^2} - G\alpha_{\text{min}}\left(2 -\frac{\gamma}{c_k}\right) \frac{3\epsilon}{4\sigma} \right)\notag.
	\end{aligned}
\end{equation}
The term in the parenthesis scales as $\mathcal{O}(1/c_k^2) - \mathcal{O}(\epsilon/c_k)$. For sufficiently small $\epsilon$, the $-\mathcal{O}(\epsilon/c_k)$ term dominates, ensuring:
\begin{equation}\label{eq:case2}
	\begin{aligned}[b]
		\min_{x^{\star} \in \mathcal{X}^{\star}} \frac{1}{n} \sum_{i=1}^n \| x_{i,k+1} - x^{\star} \|^2 \leq \min_{x^{\star} \in \mathcal{X}^{\star}} \frac{1}{n} \sum_{i=1}^n \| x_{i,k} - x^{\star} \|^2.
	\end{aligned}
\end{equation}
%the crucial step involves demonstrating that the second term in the inequality, which represents the change in the average squared distance, becomes negative for sufficiently small $\epsilon$. Let's analyze the term within the parenthesis:
%$\frac{\gamma c_0 \alpha_0}{c_k^2} -G\alpha_{\text{min}}(2 -\frac{\gamma}{c_k})\frac{3\epsilon}{4\sigma}\leq \frac{\gamma \epsilon^2}{G^2\sigma} + \frac{\alpha_{\text{min}}\gamma\epsilon}{\sqrt{\sigma \gamma c_0 \alpha_0}}\frac{3\epsilon}{4\sigma}-2G\alpha_{\text{min}} =\mathcal{O}(\epsilon^2)+ \mathcal{O}(\epsilon^3)-\mathcal{O}(\epsilon)$.
%It is evident from this asymptotic analysis that when $\epsilon$ is sufficiently small, the dominant term is $-\mathcal{O}(\epsilon)$, rendering the entire expression negative. Consequently, the term $\sigma \left( \frac{\gamma c_0 \alpha_0}{c_k^2} - G\alpha_{\text{min}}\left(2 -\frac{\gamma}{c_k}\right) \frac{3\epsilon}{4\sigma} \right)$ becomes negative. This implies that for a sufficiently small $\epsilon$, the average squared distance to an optimal solution is non-increasing:
%\[
%\min_{x^{*} \in \mathcal{X}^{*}}\frac{1}{n}\sum_{i=1}^{n}\|x_{i,k+1} - x^{*}\|^{2} \leq \min_{x^{*} \in \mathcal{X}^{*}}\frac{1}{n}\sum_{i=1}^{n}\|x_{i,k} - x^{*}\|^{2},
%\]
Combining \eqref{x bound t1} and \eqref{x bound t1,1}, there exists a final threshold $K_{\text{final}} = \max\{K_{1}^{\prime}, K_{2}^{\prime},K_{3}^{\prime},K_{4}^{\prime},\nu \}$, such that for all $k > K_{\text{final}}$, the average squared distance of the iterates to the optimal set converges to zero
\begin{align*}
	\lim_{k\rightarrow \infty}\min_{ x^{\star} \in \mathcal{X}^{\star}}\frac{1}{n}\sum_{i=1}^{n}\|x_{i,k} - x^{\star}\|^{2} = 0,
\end{align*}
this result demonstrates the convergence of the algorithm's iterates to the set of optimal solutions $\mathcal{X}^{\star}$. 
\end{IEEEproof}

\begin{theo}[Convergence rate of DPS-LA]\label{theo2}
	Suppose Assumptions  1--3 hold, and the parameter selection $c_k=\sqrt{k+1}$, the objective optimality gap for the network average $\bar{x}_k$ satisfies
	\begin{equation*}
		\min_{M \leq k \leq T} f(\bar{x}_k) - f(x^{\star}) \leq \mathcal{O}\left(\frac{1}{\sqrt{nT}}\right),
	\end{equation*}
	where $M = \lfloor T/2 \rfloor$.
	\end{theo}
\begin{IEEEproof}
This section establishes the sublinear convergence rate of the proposed algorithm. 
Rearranging the descent inequality from Lemma \ref{lemma7}, we bound the single-step optimality gap as follows
\begin{align}\label{eq:gap_rearranged}
	&\alpha_{\text{min}}(2 -\frac{\gamma}{c_k})(f(\bar{x}_k)-f(x^{\star}))\notag\\
	&\leq  \frac 1n \sum_{i=1}^n \| x_{i,k} - x^{\star} \|^2 +  G\alpha_{\text{min}}(2 -\frac{\gamma}{c_k}) \frac 1n \sum_{i=1}^n  \|z_{i,k}-\bar{x}_k\| \notag
	\\
	&\quad -\frac 1n \sum_{i=1}^n \| x_{i,k+1} - x^{\star} \|^2  + \frac{\gamma c_0 \alpha_0}{c_k^2} \frac 1n \sum_{i=1}^n   (f_i(x^{\star})-\bar{f}_i^k).
\end{align}
To bound the consensus error $\| z_{i,k} - \bar{x}_k \|$, we invoke the standard result for distributed gradients \cite{Nedic2010}
\begin{align}\label{eq:consensus_bound}
	&\| z_{i,k} - \bar{x}_k \| \notag\\
	&\leq C_1 \beta^{k-1} + C_2 \sum_{t=0}^{k-2} \beta^{k-t} \frac{\gamma c_0 \alpha_0}{c_t} + \frac{4 \gamma c_0 \alpha_0 G}{c_k^2} \triangleq \tau_k,
\end{align}
where $0 < \beta < 1$, and $C_1, C_2 > 0$ are constants related to the network topology.

Summing \eqref{eq:gap_rearranged} over $k = M, \dots, T$ and utilizing the telescoping property of the squared distances, we obtain
\begin{align}\label{eqt2}
	&\min\limits_{M\leq k \leq T}f(\bar{x}_k)-f(x^{\star}) \notag\\
	&	\leq \frac{ \sum_{i=1}^n \| x_{i,M} - x^{\star} \|^2}{\sqrt{n} \sum_{k=M}^T \alpha_{\text{min}}(2 -\frac{\gamma}{c_k})}\notag\\
	&\quad +\frac{\sum_{k=M}^TG\alpha_{\text{min}}(2 -\frac{\gamma}{c_k})\tau_k - \sum_{i=1}^n \| x_{i,T+1} - x^{\star} \|^2 }{\sqrt{n} \sum_{k=M}^T\alpha_{\text{min}}(2 -\frac{\gamma}{c_k})}\notag\\ %
	&\quad +\frac{\sum_{k=M}^T\frac{\gamma c_0 \alpha_0}{c_k^2} \frac 1n \sum_{i=1}^n   (f_i(x^{\star})-\bar{f}_i^k) }{\sqrt{n} \sum_{k=M}^T\alpha_{\text{min}}(2 -\frac{\gamma}{c_k})}.
\end{align}
We now analyze the asymptotic behavior of each component in \eqref{eqt2} term by term. For the case $c_k = \sqrt{k+1}$, we have
\begin{equation}\label{eq1t2}
	\sqrt{n} \sum_{k=M}^T\alpha_{\text{min}}(2 -\frac{\gamma}{c_k})
	\leq 	\sqrt{n} \sum_{k=M}^T  \frac{c_0\alpha_0}{2\sqrt{k+1}} = \mathcal{O}(\sqrt{nT}),
\end{equation}
the initial average squared distance, $\| x_{i,M} - x^{\star} \|^2$, is bounded by a constant, hence it is $\mathcal{O}(1)$.

For the term related to the optimality gap of individual agents and given $M=\lfloor T/2 \rfloor$
\begin{equation}\label{eq2t2}
	\begin{aligned}[b]
		&\sum_{k=M}^T \frac{\gamma c_0 \alpha_0}{c_k^2} \sum_{i=1}^n   (f_i(x^{\star})-\bar{f}_i^k)
		\\&\leq 	\sum_{k=M}^T\frac{\gamma c_0 \alpha_0}{k+1} \left( \sum_{i=1}^n   {f_i^{\star}}- \sum_{i=1}^n  \bar{f_i^0}\right) = \mathcal{O}(1),
	\end{aligned}
\end{equation}
next, we analyze the exponential decay term from $\tau_k$
\begin{equation}\label{eq3t2}
	\begin{aligned}[b]
		&\sum_{k=M}^T	\alpha_{\text{min}}(2 -\frac{\gamma}{c_k})\beta^{k-1}\\
		&\leq	\sum_{k=M}^T \frac{c_0\alpha_0}{2\sqrt{k+1}}(2 -\frac{\gamma}{c_k})\beta^{k-1}\\
		&\leq \sum_{k=M}^T \left(\frac{c_0\alpha_0}{\sqrt{k+1}}-\frac{\gamma c_0\alpha_0}{2(k+1)}\right)\beta^{k-1}=\mathcal{O}\left(\frac{1}{\sqrt{T}}\right).
	\end{aligned}
\end{equation}
The \eqref{eq3t2} rapid decay is due to the geometric progression factor $\beta^{k-1}$. We assume $M=\lfloor T/2 \rfloor$ for this analysis. For the last and most complex term within $\tau_k$, which involves a nested sum
\begin{align*}
	& \sum_{k=M}^{T}\alpha_{\text{min}}(2 -\frac{\gamma}{c_k})\sum_{t=0}^{k-2}\beta^{k-t}\frac{\gamma c_0 \alpha_0}{c_t}\\
	&\leq\sum_{k=M}^{T} \frac{c_0 \alpha_0}{\sqrt{k+1}}\sum_{t=0}^{k-2}\beta^{k-t}\frac{\gamma c_0 \alpha_0}{\sqrt{t+1}},
\end{align*}
we first consider the inner factor $\sum_{t=0}^{k-2}\beta^{k-t}\frac{1}{\sqrt{t+1}}$. Let $p=k-2$, and we split this sum into two parts
\begin{align*}
	&\sum_{t=0}^{k-2}\beta^{k-t}\frac{1}{\sqrt{t+1}}\\
	&=\beta^{2}\sum_{t=0}^{k-2}\beta^{k-2-t}\frac{1}{\sqrt{t+1}}\\
	&=\beta^{2}\sum_{t=0}^{p}\beta^{p-t} \frac{1}{\sqrt{t+1}}\\
	&=\beta^{2}\left[\sum_{t=0}^{\lfloor p/2\rfloor}\beta^{p-t}\frac{1}{\sqrt{t+1}}+\sum_{t=\lfloor p/2\rfloor+1}^{p}\beta^{p-t}\frac{1}{\sqrt{t+1}}\right],
\end{align*}
for the first term, $\sum_{t=0}^{\lfloor p/2\rfloor}\beta^{p-t}\frac{1}{\sqrt{t+1}}$, we have
\[
\sum_{t=0}^{\lfloor p/2\rfloor}\beta^{p-t}\frac{1}{\sqrt{t+1}} \leq\beta^{p-\lfloor p/2\rfloor}\sum_{t=0}^{\lfloor p/2\rfloor} \frac{1}{\sqrt{t+1}}
\]
using an integral approximation for the sum $\sum_{t=0}^{\lfloor p/2\rfloor} \frac{1}{\sqrt{t+1}}$:
\begin{align*}
	&\sum_{t=0}^{\lfloor p/2\rfloor}\beta^{p-t}\frac{1}{\sqrt{t+1}}\\
	&\leq\beta^{p/2} \int_{0}^{\lfloor p/2\rfloor}\frac{1}{\sqrt{t+1}}dt = \beta^{p/2} \left[2\sqrt{t+1}\right]_{0}^{\lfloor p/2\rfloor} \\
	&= 2\beta^{p/2}(\sqrt{\lfloor p/2\rfloor+1}-1),
\end{align*}
it is easy to verify that for any $1>\beta_{1}>\sqrt{\beta}$
\[
\lim_{p\rightarrow\infty}\frac{\beta^{p/2}\left(\sqrt{\lfloor p/2\rfloor+1}-1 \right)}{\beta_{1}^{p}}=0.
\]
we thus have that $\sum_{t=0}^{\lfloor p/2\rfloor}\beta^{p-t}\frac{1}{\sqrt{t+1}}\leq 2C_{\beta_{1}}\beta_{1}^{p}$, where $C_{\beta_{1}}$ is a constant.
next, we consider the term $\sum_{t=\lfloor p/2\rfloor+1}^{p}\beta^{p-t}\frac{1}{\sqrt{t+1}}$. We bound it as
\[
\sum_{t=\lfloor p/2\rfloor+1}^{p}\beta^{p-t}\frac{1}{\sqrt{t+1}} \leq\frac{1}{\sqrt{\lfloor p/2\rfloor+2}}\sum_{t=\lfloor p/2\rfloor+1}^{p} \beta^{p-t}
\]
\[
=\frac{1}{\sqrt{\lfloor p/2\rfloor+2}}\beta^{2\lfloor p/2\rfloor -p+2}\frac{1-\beta^{p-\lfloor p/2\rfloor}}{1-\beta}.
\]
as $1\leq 2\lfloor p/2\rfloor-p+2\leq 2$ and $\frac{1-\beta^{p-\lfloor p/2\rfloor}}{1-\beta}\leq\frac{1}{1-\beta}$, there exists a constant $C_{\beta}$ such that 
\[
\sum_{t=\lfloor p/2\rfloor+1}^{p}\beta^{p-t}\frac{1}{\sqrt{t+1}}\leq C_{\beta }\frac{1}{\sqrt{\lfloor p/2\rfloor+2}}.
\]
noticing $p=k-2$, we substitute these bounds back into the outer sum
\begin{equation}\label{eq4t2}
	\begin{aligned}
		&\sum_{k=M}^{T}\alpha_{\text{min}}(2 -\frac{\gamma}{c_k})\sum_{t=0}^{k-2}\beta^{k-t}\frac{\gamma c_0 \alpha_0}{c_t} \\
		&\leq 2C_{\beta_{1}}\sum_{k=M}^{T}\frac{\gamma c_0^2 \alpha_0^2}{2\sqrt{k+1}}\beta_{1}^{(k-2)}+C_{\beta}\sum_{k=M}^{T}\frac{\sqrt{2}\gamma c_0^2 \alpha_0^2}{2\sqrt{(k+1)(k+3)}} \\
		&\leq \gamma c_0^2 \alpha_0^2(C_{\beta_{1}}\sum_{k=M}^{T}\frac{1}{\sqrt{k+1}}\beta_{1}^{(k-2)}+\frac{\sqrt{2}}{2}C_{\beta}\sum_{k=M}^{T}\frac{1}{k+1})
	\end{aligned}
\end{equation}
the first term in this sum is $\mathcal{O}\left(\frac{1}{\sqrt{T}}\right)$ due to the exponential decay $\beta_1^{(k-2)}$. The second term, $\sum_{k=M}^{T}\frac{1}{k+1}$, is $\mathcal{O}(1)$ when $M=\lfloor T/2 \rfloor$.
Thus, this entire nested sum term is $\mathcal{O}(1)$.

Combining the asymptotic bounds for the numerator terms from \eqref{eq1t2}, \eqref{eq2t2}, \eqref{eq3t2}, and \eqref{eq4t2}, 	we observe that the dominant terms are of order $\mathcal{O}(1)$.
Thus, the convergence rate satisfies:
\[
\min\limits_{M\leq k\leq T}f(\bar{x}_k)-f(x^{\star})\leq \mathcal{O}\left(\frac{1}{\sqrt{nT}}\right).
\]
this result shows that the algorithm converges to the optimal solution at a sublinear rate, with the optimality gap shrinking as $T$ increases. 
\end{IEEEproof}

\begin{remark}
	 The Theorem \ref{theo1} confirms the overall convergence properties of the algorithm, demonstrating that the selection of parameters does not affect the final convergence outcome. Finally, Theorem  \ref{theo2} shows that for specific values of the sequence \( c_k \), the algorithm achieves a convergence rate of \( \mathcal{O}(1/\sqrt{nT}) \).
\end{remark}
\section{Numerical Experiments}\label{sec5}

This section evaluates the Algorithm \ref{algorithm1} through numerical simulations in a distributed optimization environment. We consider a multi-agent system consisting of $N=4$ agents collaborating to solve a common optimization problem over a decision variable $\theta \in \mathbb{R}^d$ with dimension $d=6$. Each agent $i$ has a local objective function defined by a quadratic loss,
\[
f_i(\theta) = \frac{1}{2}\lVert A_i\theta - b_i \rVert^2,
\]
where elements of the data matrices $A_i \in \mathbb{R}^{2 \times 6}$ and vectors $b_i \in \mathbb{R}^2$are randomly sampled from uniform distributions $U(0, 0.1)$ and $U(0, 5)$, respectively. A key feature of our setup is the definition of a non-trivial, common constraint set for all agents. All agents are subject to the same box constraint, i.e., $\theta \in \mathcal{X} = [\mathbf{l}, \mathbf{u}]$. These bounds are systematically generated based on the unconstrained global minimizer, $\theta_{\text{unc}}^{\star}$. Specifically, for each dimension $j \in \{1, \dots, 6\}$, the common lower and upper bounds are defined as
$
	l_{j} = \theta_{\text{unc},j}^{\star} + 10 + 10 \sin\left(\frac{j\pi}{120}\right), 
	u_{j} = \theta_{\text{unc},j}^{\star} + 20 + 10 \sin\left(\frac{j\pi}{120}\right).
$
This formulation creates a single, shared feasible set $\mathcal{X}$ within which the optimal solution must be found. The agents communicate over a connected, undirected random graph. The corresponding doubly stochastic weight matrix $W$ is constructed using the Metropolis-Hastings rule to govern the consensus updates. The network structure considered in the experiments is shown in Fig.~\ref{graph}. For Algorithm \ref{algorithm1}, the parameters for each agent are configured as $\gamma = 1.0$ and $\bar{\gamma} = 1.5$, with the stepsize rule incorporating a scaling factor of $c_k = 0.5\sqrt{k+1}$. The initial level-value is uniformly set to $-500$ across all agents.  We compare its performance against the DGD, which employs a conventional diminishing stepsize rule of $\alpha_k = \frac{2}{k+1}$.
% The ground-truth global optimal solution $\theta^{\star}$ is obtained by solving the centralized constrained optimization problem $\min_{\theta \in \mathcal{X}} \sum_{i=1}^N f_i(\theta)$.
%This solution serves as the benchmark for performance evaluation. 
The primary metric used to assess the algorithms is the function value error, defined as $
\text{Resdual errors}(k) = f(\bar{\theta}_k) - f(\theta^{\star}),
$
where $f(\cdot) = \sum_{i=1}^N f_i(\cdot)$ and $\bar{\theta}_k = \frac{1}{N}\sum_{i=1}^N \theta_{i,k}$ . Additional analyses on consensus error, level-value evolution, and stepsize dynamics are also performed to provide a comprehensive evaluation.
\begin{figure}[t]
	\centering
	\subfloat{
		\includegraphics[width=0.5\linewidth]{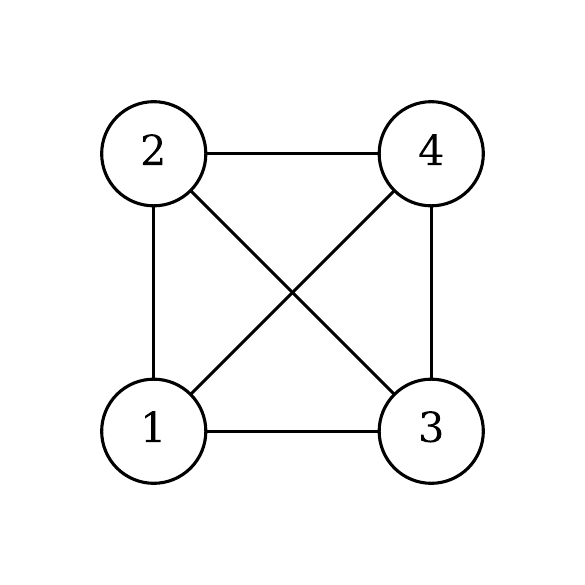} % 比 0.5 更宽就不容易换行
	}
	\vspace*{-2em}
	\caption{Network graph of numerical simulation.}
	\label{graph}
\end{figure}

From Fig.~\ref{fig:performance_comparison}(a), we observe that the proposed algorithm demonstrates a substantially superior convergence rate compared to the DGD algorithm. Specifically, the function error of the proposed algorithm decreases precipitously within the initial 50 iterations, rapidly reaching a near-zero steady-state value. In stark contrast, the DGD algorithm exhibits a much slower, gradual reduction in function error, failing to achieve a comparable level of accuracy even after 300 iterations. 
\begin{figure}[t]
	\centering
	
	% 第一张图 - 单独一行
	\subfloat{%
		\includegraphics[width=0.8\linewidth]{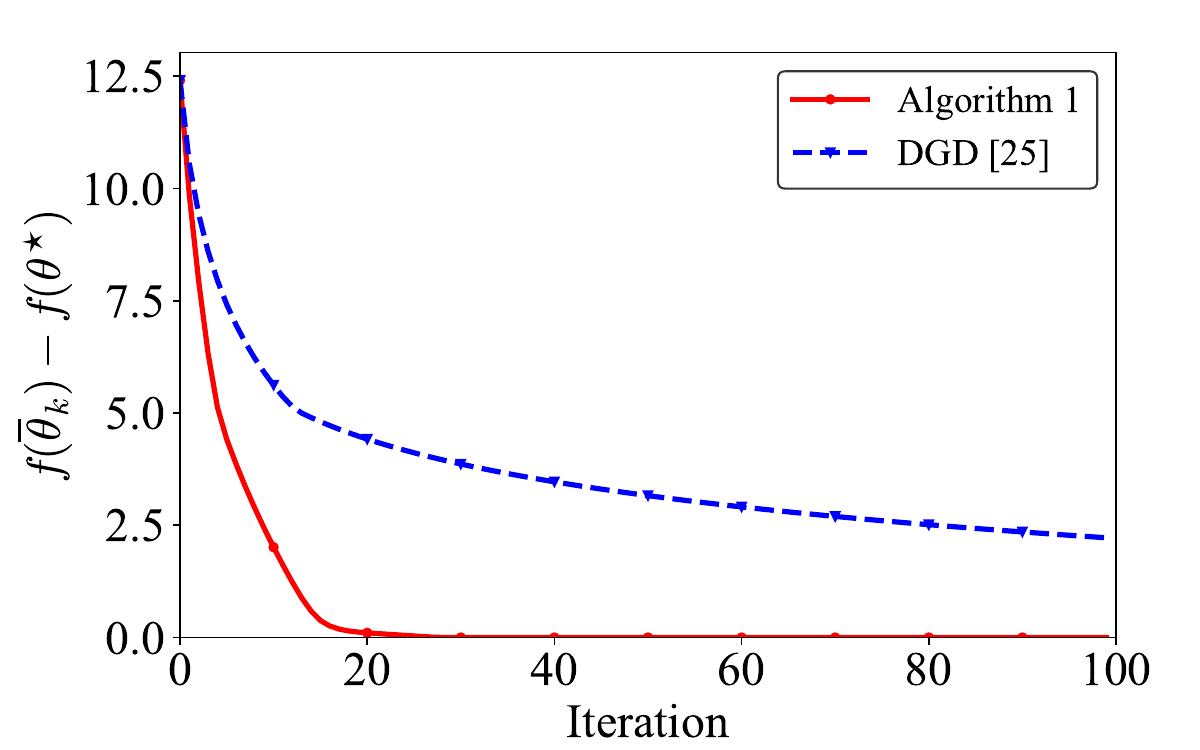}} \\
	
	% 第二张图 - 单独一行

	\caption{Residual errors comparison of the proposed algorithm and DGD algorithm.}
	\label{fig:performance_comparison}
\end{figure}
As illustrated in Fig. \ref{fig:performance_comparison2}, the estimated level-values $\bar{f}_i^k$ converge rapidly and accurately to the true optimal values $f_i(\theta^{\star})$. Furthermore, Fig.\ref{fig:performance_comparison2}(b) confirms that consensus among agents is also achieved rapidly, contributing to the overall stability and efficiency of the proposed algorithm. This precise estimation enables the algorithm to adaptively compute a more aggressive and near-optimal stepsize, as shown in Fig. \ref{fig:liner speedup}(a), which significantly accelerates the reduction of the optimality gap. Finally, the results in Fig. \ref{fig:liner speedup}(b) demonstrate that increasing the number of agents improves the convergence rate, thereby validating the theoretical linear speedup property.

\begin{figure}[t]
	\centering
		\subfloat[Level-values adjustment]{%
		\includegraphics[width=0.8\linewidth]{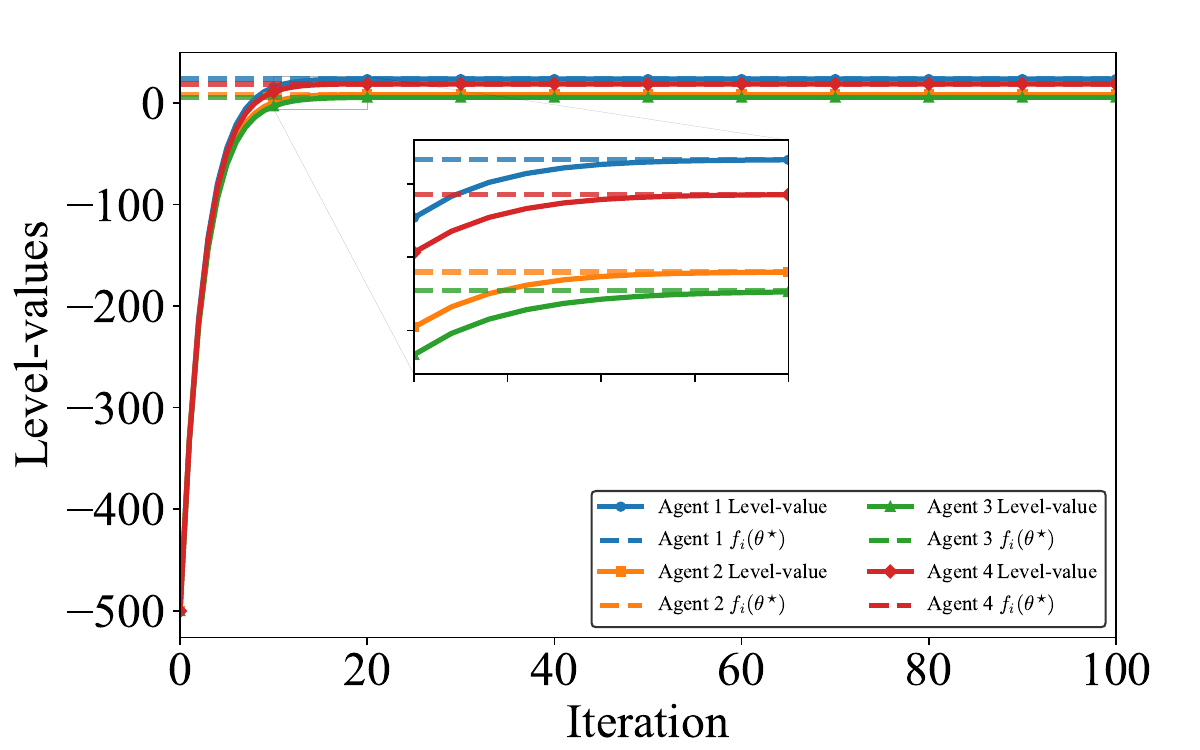}} \\
		
		\subfloat[Consensus errors]{%
		\includegraphics[width=0.8\linewidth]{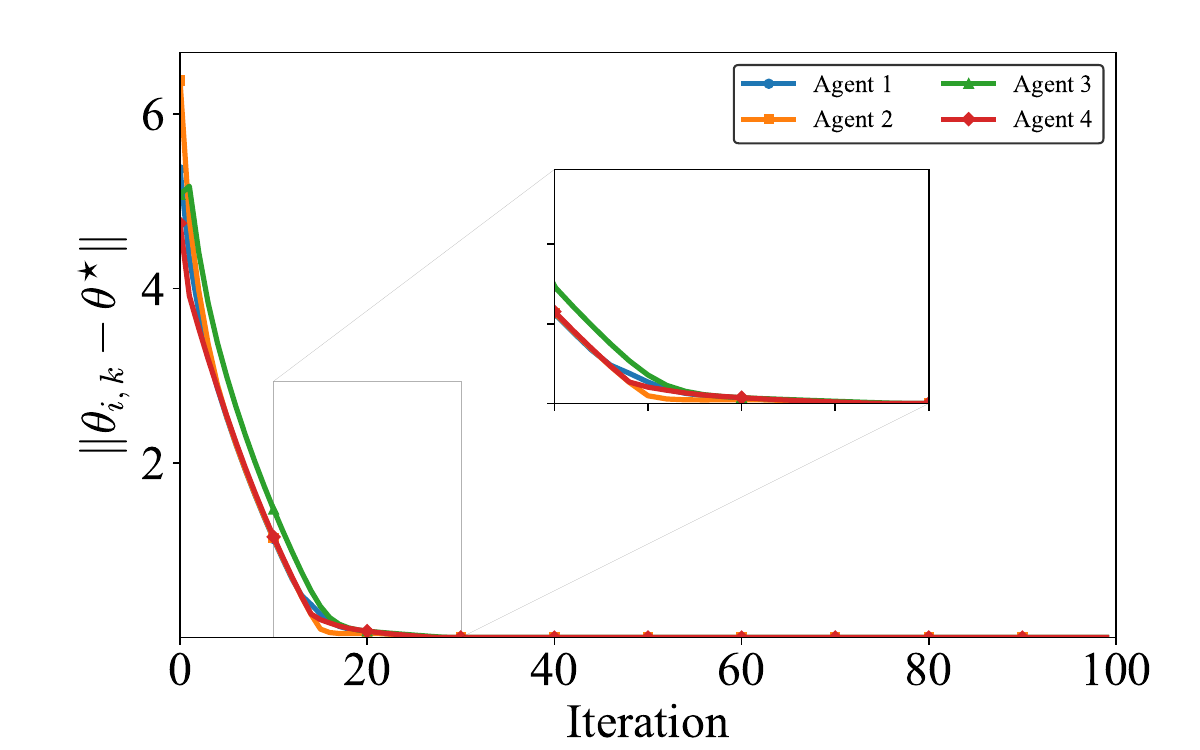}} \\
	\caption{Evolution of level-values and consensus error for each agent.}
	\label{fig:performance_comparison2}
\end{figure}
% TODO: \usepackage{graphicx} required
%\begin{figure}
%	\centering
%	\includegraphics[width=1.05\linewidth]{../1/ex}
%	\vspace{-1em}
%	\caption{Performance comparison of the proposed algorithm and DGD algorithm.}
%	\label{fig:performance_comparison}
%\end{figure}

%\begin{figure}[t]
%	\centering
%	
%	% 第一张图 - 单独一行
%	\subfloat[Residual errors]{%
%		\includegraphics[width=0.7\linewidth]{function_error.pdf}} \\
%	
%	% 第二张图 - 单独一行
%	\subfloat[Consensus errors]{%
%		\includegraphics[width=0.7\linewidth]{distance_to_optimal.pdf}} \\
%	
%	% 第三张图 - 单独一行
%	\subfloat[Level-values adjustment]{%
%		\includegraphics[width=0.7\linewidth]{level_value_evolution.pdf}} \\
%	
%	% 第四张图 - 单独一行
%	\subfloat[Agents stepsizes]{%
%		\includegraphics[width=0.7\linewidth]{stepsize_evolution.pdf}}
%	
%	\caption{Performance comparison of the proposed algorithm and DGD algorithm.}
%	\label{fig:performance_comparison}
%\end{figure}
\begin{figure}[t]
	% \raggedright 去掉，避免和 \centering 打架
	\centering
	\subfloat[Adaptive stepsizes of each agent]{%
	\includegraphics[width=0.8\linewidth]{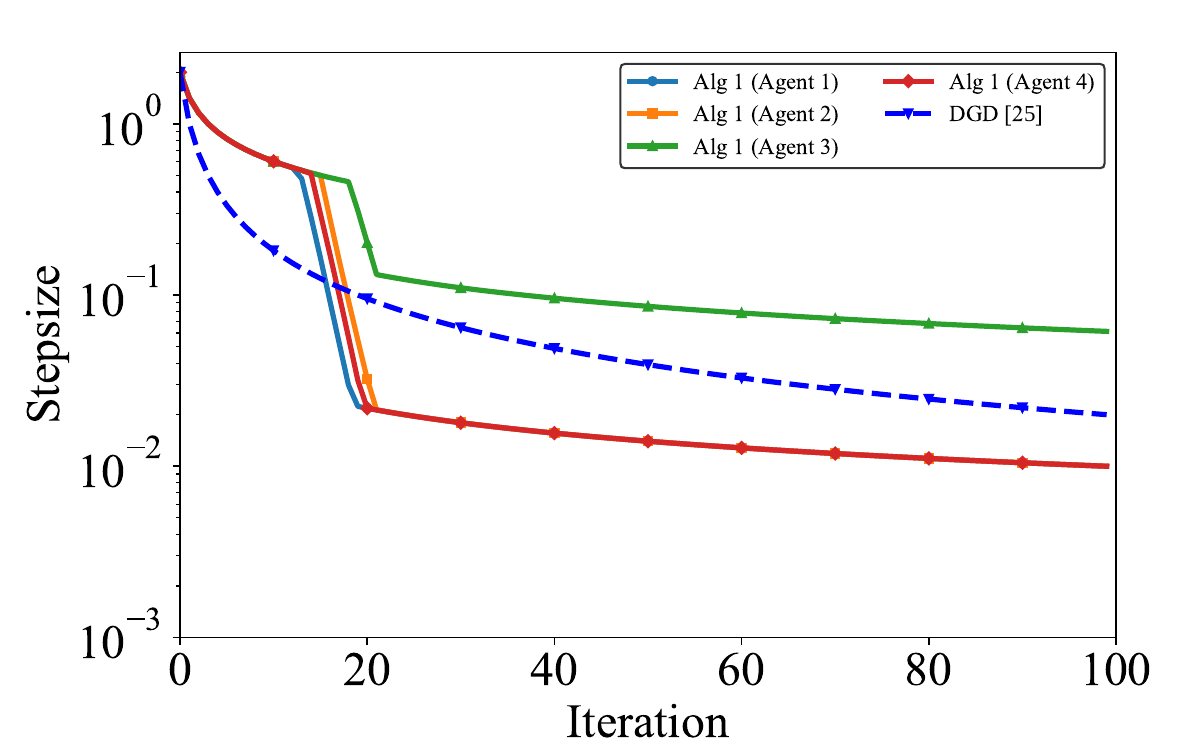}}
	
\subfloat[Different number of agents]{
	\includegraphics[width=0.8\linewidth]{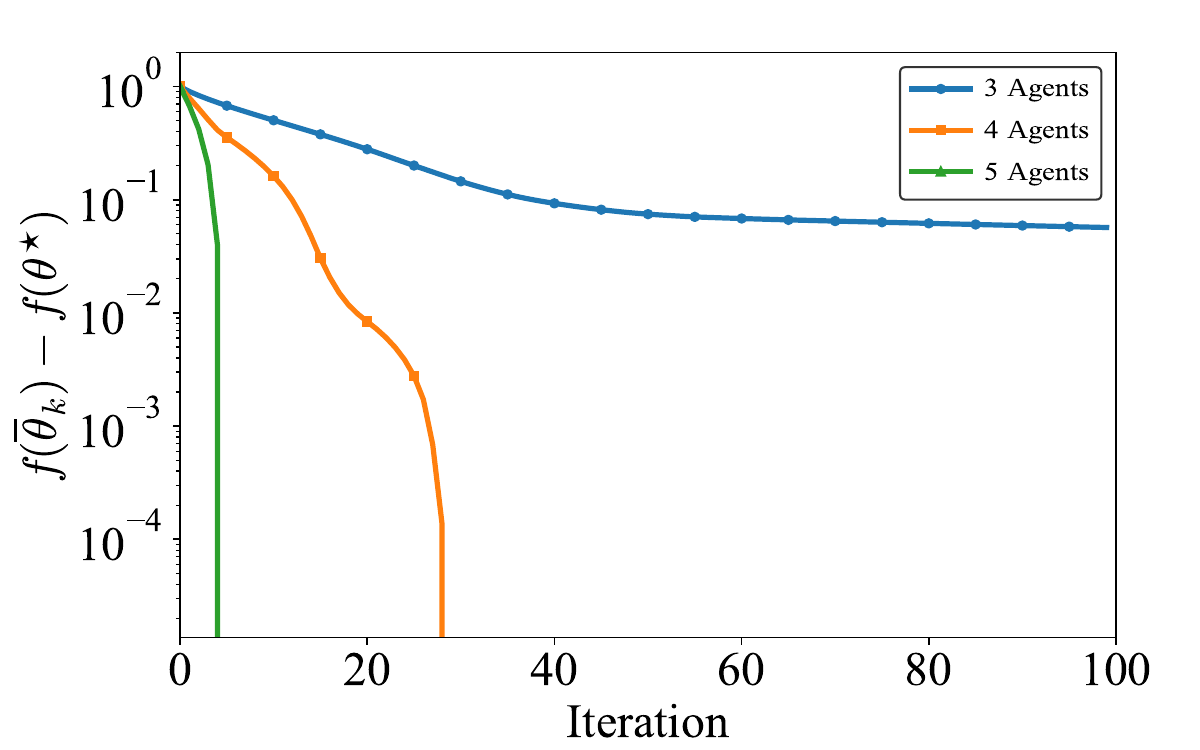}
}
\caption{Adaptive stepsizes of each agent and performance comparison with different numbers of agents.}
\label{fig:liner speedup}
\end{figure}

\section{Conclusion}\label{sec6}
In this work, we have investigated the distributed constrained optimization problem. A distributed gradient algorithm, incorporating a Polyak stepsize with a level-value adjustment mechanism, has been proposed. This design bypasses the conventional requirement for prior knowledge of the global optimal value. We have rigorously shown that the proposed algorithm achieves a convergence rate of $\mathcal{O}({\frac{1}{\sqrt{nT}}})$ in terms of the objective optimality gap. Both theoretical analysis and numerical experiments establish the efficacy. Future works can focus on integrating the proposed algorithm with acceleration techniques, such as gradient tracking and EXTRA, to further enhance convergence performance in diverse network environments.

\appendix

%\begin{lemma}
%	Let $\{c_k\}_{k \ge 0}$ be a non-decreasing sequence of positive real numbers. Then, the step-size $\alpha_{i,k}$ is bounded and decays as the iteration $k$ increases. Moreover, the sequence $\{\alpha_{i,k}\}_{k \ge 0}$ is non-increasing, i.e., $\alpha_{i,k} \leq \alpha_{i,k-1}$ for all $i \in V$ and $k$.
%\end{lemma}

	\subsection{Proof of Lemma \ref{bound}}\label{L2proof}
\begin{IEEEproof}
	Firstly, since $\alpha_{i,k} \leq \frac{c_{k-1}}{c_k}\alpha_{i,k-1}$ and $c_k$ is non-decreasing, we have that $\alpha_{i,k}$ is non-increasing.
	
	For $k=0$:
	\begin{equation*}
		\alpha_{i,0} =  \frac{1}{c_0}\min\left\{ \max\left\{ \frac{c_0\alpha_0}{2}, \beta_{i,0} \right\}, c_0\alpha_0 \right\},
	\end{equation*}
	where $\beta_{i,0} = \gamma \frac{f_i(z_{i,0}) - \bar{f}^0_i}{\|\nabla f_i(z_{i,0})\|^2}$. 
	
	If $\beta_{i,0} \leq \frac{c_0\alpha_0}{2}$, then $\max\left\{ \frac{c_0\alpha_0}{2}, \beta_{i,0} \right\} = \frac{c_0\alpha_0}{2}$, and thus $\alpha_{i,0} = \frac{1}{c_0}  \min\left\{ \frac{c_0\alpha_0}{2}, c_0\alpha_0 \right\} = \frac{\alpha_0}{2}$. If $\beta_{i,0} > \frac{c_0\alpha_0}{2}$, then $\max\left\{ \frac{c_0\alpha_0}{2}, \beta_{i,0} \right\} = \beta_{i,0}$, and thus $\alpha_{i,0} = \frac{1}{c_0} \cdot \min\left\{ \beta_{i,0}, c_0\alpha_0 \right\}$. If in addition $\beta_{i,0} \geq c_0\alpha_0$, then $\alpha_{i,0} = \alpha_0$; otherwise if $\beta_{i,0} < c_0\alpha_0$, then $\alpha_{i,0} = \frac{\beta_{i,0}}{c_0} \in \left( \frac{\alpha_0}{2}, \alpha_0 \right)$. 
	
	In all cases, we have $\alpha_{i,0} \in \left[ \frac{\alpha_0}{2}, \alpha_0 \right] = \left[ \frac{c_0\alpha_0}{2c_0}, \frac{c_0\alpha_0}{c_0} \right]$. Hence, the base case holds. Now assume by induction that for some $k$,
	\begin{equation*}
		\frac{c_0\alpha_0}{2c_k} \leq \alpha_{i,k} \leq \frac{c_0\alpha_0}{c_k},
	\end{equation*}
for $k+1$, we have
	\begin{equation*}
		\alpha_{i,k+1} = \frac{1}{c_{k+1}}  \min\left\{ \max\left\{ \frac{c_0\alpha_0}{2}, \beta_{i,k+1} \right\}, c_k\alpha_{i,k} \right\},
	\end{equation*}
	from the induction hypothesis, $c_k\alpha_{i,k} \in \left[ \frac{c_0\alpha_0}{2}, c_0\alpha_0 \right]$. Let $t = \frac{c_k\alpha_{i,k}}{c_0} \in [\alpha_0/2, \alpha_0]$.
	
	Consider first the case when $\beta_{i,k+1} \leq \frac{c_0\alpha_0}{2}$. Then $\max\left\{ \frac{c_0\alpha_0}{2}, \beta_{i,k+1} \right\} = \frac{c_0\alpha_0}{2}$, and 
	\begin{align*}
		\alpha_{i,k+1} &= \frac{1}{c_{k+1}}  \min\left\{ \frac{c_0\alpha_0}{2}, c_k\alpha_{i,k} \right\} \\
		&= \frac{1}{c_{k+1}}  \min\left\{ \frac{c_0\alpha_0}{2}, c_0t \right\} \\
		&= \frac{c_0}{c_{k+1}}  \min\left\{ \frac{\alpha_0}{2}, t \right\},
	\end{align*}
	since $t \geq \frac{\alpha_0}{2}$, we have $\min\left\{ \frac{\alpha_0}{2}, t \right\} = \frac{\alpha_0}{2}$, thus $\alpha_{i,k+1} = \frac{c_0\alpha_0}{2c_{k+1}}$. Now consider the case when $\beta_{i,k+1} > \frac{c_0\alpha_0}{2}$. Then $\max\left\{ \frac{c_0\alpha_0}{2}, \beta_{i,k+1} \right\} = \beta_{i,k+1}$, and 
	\begin{align*}
		\alpha_{i,k+1} &= \frac{1}{c_{k+1}} \min\left\{ \beta_{i,k+1}, c_k\alpha_{i,k} \right\} \\
		&= \frac{1}{c_{k+1}}  \min\left\{ \beta_{i,k+1}, c_0t \right\},
	\end{align*}
	if $\beta_{i,k+1} \geq c_0t$, then $\alpha_{i,k+1} = \frac{c_0t}{c_{k+1}} \in \left[ \frac{c_0\alpha_0}{2c_{k+1}}, \frac{c_0\alpha_0}{c_{k+1}} \right]$. 
	
	If $\beta_{i,k+1} < c_0t$, then $\alpha_{i,k+1} = \frac{\beta_{i,k+1}}{c_{k+1}}$. Since $\beta_{i,k+1} > \frac{c_0\alpha_0}{2}$, we have $\alpha_{i,k+1} > \frac{c_0\alpha_0}{2c_{k+1}}$, and since $\beta_{i,k+1} \leq c_0t \leq c_0\alpha_0$, we have $\alpha_{i,k+1} \leq \frac{c_0\alpha_0}{c_{k+1}}$. In all cases, we have shown that $\frac{c_0\alpha_0}{2c_{k+1}} \leq \alpha_{i,k+1} \leq \frac{c_0\alpha_0}{c_{k+1}}$.
\end{IEEEproof}

	\subsection{Proof of Lemma \ref{psvd}}\label{L3proof}
\begin{IEEEproof}
	Since the constraint system \eqref{12} admits no feasible solution, \( x^{\star} \in \mathcal{X}^{\star} \) is not feasible to \eqref{12}. then the optimal point $x^{\star} \in \mathcal{X}^{\star}$ is necessarily infeasible to \eqref{12}. Consequently, there must exist at least one index $k \in \{k_i(j), k_i(j)+1, \ldots, k_i(j)+\eta\}$ that violates the constraint, such that
  \begin{equation*}
	\left( \nabla f_i^k \right)^\text{T} (x^{\star} - z_{i,k}) > - \frac{1}{\bar{\gamma}} \beta_{i,k} \| \nabla f_i^k \|^2.
\end{equation*}
Rearranging the terms yields a lower bound on the $\beta_{i,k}$ as follows
    \begin{equation}\label{proof_step1}
	\beta_{i,k} > \frac{{\bar{\gamma}} \left( \nabla f_i^k \right)^\text{T} \left( z_{i,k} - x^{\star} \right)}{\| \nabla f_i^k \|^2}.
\end{equation}
By invoking the first-order condition of convexity, i.e., $f_i(x^\star) \geq f_i(z_{i,k}) + (\nabla f_i^k)^\text{T}(x^\star - z_{i,k})$, \eqref{proof_step1} can be simplified to
\begin{equation*}
	\beta_{i,k} > \bar{\gamma} \frac{f_i(z_{i,k}) - f_i^{\star}}{\| \nabla f_i^k\|^2}.
\end{equation*}
  Applying the $\beta_{i,k} = \gamma (f_i(z_{i,k}) - \bar{f}_i^k) / \|\nabla f_i^k\|^2$, we obtain
\begin{equation*}
	\gamma \frac{f_i(z_{i,k}) - \bar{f}_i^k}{\|\nabla f_i^k\|^2} > \bar{\gamma} \frac{f_i(z_{i,k}) - f_i^{\star}}{\|\nabla f_i^k\|^2}.
	\end{equation*}
		By canceling the common term $\|\nabla f_i^k\|^2$ and solving for $f_i^{\star}$, we arrive at:
canceling \(\|\nabla f_i^k\|^2\) and rearranging terms yields
    \begin{equation*}
	f_i^{\star} > \left(1 - \frac{\gamma}{\bar{\gamma}}\right) f_i(z_{i,k}) + \frac{\gamma}{\bar{\gamma}} \bar{f}_i^k.
\end{equation*}
 The right-hand side of the above inequality precisely defines the updated level-value $\bar{f}_i^{k+1}$ as
\begin{equation}\label{14}
	\bar{f}_i^{k+1} = \frac{\gamma}{\bar{\gamma}} \bar{f}_i^k + \left(1 - \frac{\gamma}{\bar{\gamma}}\right) f_i(z_{i,k}).
\end{equation}
Since $\bar{\gamma} > \gamma > 0$, the update in \eqref{14} represents a convex combination of $\bar{f}_i^k$ and $f_i(z_{i,k})$. Given that $\bar{f}_i^k < f_i(z_{i,k})$ holds by construction, the sequence of level-values satisfies the strict monotonicity and bounding property
\begin{equation*}
	\bar{f}_i^k < \bar{f}_i^{k+1} < f_i^{\star},
\end{equation*}
which implies that $\bar{f}_i^{k+1}$ serves as a tighter lower bound for the optimal value $f_i^{\star}$.
\end{IEEEproof}
	\subsection{Proof of Lemma \ref{lemma 7}}\label{L4proof}
\begin{IEEEproof}
We establish the bound by analyzing the evolution of the squared distance for a single agent and then aggregating the results across the network.
 
By the non-expansiveness property of the projection operator $P_{\mathcal{X}}$, it follows that $\| P_{\mathcal{X}}(u) - P_{\mathcal{X}}(v) \| \leq \| u - v \|$. Setting $v = x^{\star} = P_{\mathcal{X}}(x^{\star})$, we have
	\begin{align}\label{x_bound_refined}
		&\| x_{i,k+1} - x^{\star} \|^2\notag\\
		& \leq \| (z_{i,k} - \alpha_{i,k} \nabla f_{i}^k) - x^{\star} \|^2\notag\\
		&= \| z_{i,k} - x^{\star} \|^2 - 2\alpha_{i,k} \langle \nabla f_i^k, z_{i,k} - x^{\star} \rangle + \alpha_{i,k}^2 \| \nabla f_i^k \|^2.
	\end{align}
Averaging \eqref{x_bound_refined} over all $n$ agents and utilizing the consensus property $\frac{1}{n} \sum_{i=1}^n \| z_{i,k} - x^{\star} \|^2 \leq \frac{1}{n} \sum_{i=1}^n \| x_{i,k} - x^{\star} \|^2$ (which holds under Assumption \ref{a1}), we obtain
\begin{equation*}
	\begin{aligned}
		& \quad \| x_{i,k+1} - x^{\star} \|^2 \\
		&  \quad \leq \| z_{i,k} - x^{\star} \|^2 - 2\alpha_{i,k} (f(z_{i,k})-f_i(x^{\star})) + \alpha_{i,k}^2 \| \nabla f_i^k \|^2\\
		&\quad \leq \| z_{i,k} - x^{\star} \|^2 - 2\alpha_{i,k} (f(z_{i,k})-f_i(x^{\star})) \\
		& \quad \quad+ \frac{\gamma}{c_k} \alpha_{i,k}( f(z_{i,k})-\bar{f_i}^k)\\
		& \quad= \| z_{i,k} - x^{\star} \|^2 - 2\alpha_{i,k} (f(z_{i,k})-f_i(x^{\star})) \\
		&\quad \quad+ \frac{\gamma}{c_k} \alpha_{i,k} (f(z_{i,k})-f_i(x^{\star}) + f_i(x^{\star})-\bar{f_i}^k) \\
		&\quad = \| z_{i,k} - x^{\star} \|^2 - (2 -\frac{\gamma}{c_k})\alpha_{i,k}(f(z_{i,k})-f_i(x^{\star})) \\
		&\quad \quad+ \frac{\gamma}{c_k} \alpha_{i,k} (f_i(x^{\star})-\bar{f}_i^k)).
	\end{aligned}
\end{equation*}
To introduce the average state $\bar{x}_k$, we decompose the function gap term as $f_i(z_{i,k}) - f_i(x^{\star}) = [f_i(z_{i,k}) - f_i(\bar{x}_k)] + [f_i(\bar{x}_k) - f_i(x^{\star})]$. 
Since $\mathcal{X}$ is compact, $f_i$ is $G$-Lipschitz continuous, implying $f_i(\bar{x}_k) - f_i(z_{i,k}) \leq G \| z_{i,k} - \bar{x}_k \|$. Substituting this into the average bound yields
\begin{align*}
	&\frac{1}{n} \sum_{i=1}^n \| x_{i,k+1} - x^{\star} \|^2 - \frac{1}{n} \sum_{i=1}^n \| x_{i,k} - x^{\star} \|^2 \\
	&\leq G \alpha_{\text{min}} \left(2 - \frac{\gamma}{c_k}\right) \frac{1}{n} \sum_{i=1}^n \| z_{i,k} - \bar{x}_k \| \\
	&\quad - \alpha_{\text{min}} \left(2 - \frac{\gamma}{c_k}\right) (f(\bar{x}_k) - f(x^{\star})) \\
	&\quad + \frac{\gamma c_0 \alpha_0}{c_k^2} \frac{1}{n} \sum_{i=1}^n (f_i(x^{\star}) - \bar{f}_i^k).
\end{align*}
This completes the proof.
\end{IEEEproof}

%\begin{lemma}\label{lemma8}
%Suppose Assumptions 1-3 hold, and with the parameter $c_k=\sqrt{k+1}$, it can be shown that there exists a subsequence $\{y(k_\nu)\}$ of the iterates $\{\bar{x}_k\}$ such that $\lim_{\nu \to \infty} f(y(k_\nu)) - f(x^{\star}) = 0$. For convenience in the subsequent analysis, we define $\sigma = f(x^{\star}) - \bar{f}^0$, where $f(x^{\star})$ is the global optimal value and $\bar{f}^0$ is the initial common level-value for all agents.
%\end{lemma}
	\subsection{Proof of Lemma \ref{lemma8}}\label{L5proof}
	\begin{IEEEproof}
 we first prove by contradiction that  
\[
\liminf_{k \to \infty} f(\bar{x}_k) - f(x^{\star}) \leq 0.
\]
Suppose there exist $\epsilon > 0$ and $K_\epsilon$ such that $f(\bar{x}_k) - f(x^{\star}) \geq \epsilon$ for all $k \geq K_\epsilon$.
\[
f(\bar{x}_k) - f(x^{\star}) \geq \epsilon,
\]
from Lemma \ref{lemma 7}, we have
\begin{align}
	&\frac 1n \sum_{i=1}^n \| x_{i,k+1} - x^{\star} \|^2  
-  \frac 1n \sum_{i=1}^n \| x_{i,k} - x^{\star} \|^2 \notag \\
&\leq  G\alpha_{\text{min}}(2 -\frac{\gamma}{c_k})(\frac 1n \sum_{i=1}^n  \|z_{i,k}-\bar{x}_k\|-\epsilon)\notag \\
	&\quad + \frac{\gamma c_0 \alpha_0}{c_k^2} \frac 1n \sum_{i=1}^n   (f_i(x^{\star})-\bar{f_i^k})),\notag
\end{align}	
when the $c_k=\sqrt{k+1}$, which implies decaying stepsizes for \(\alpha_{i,k}\), there exists a constant \(K'_1 \in \mathbb{N}^+\) such that for \(k > K'_1\), the \(c_k\) satisfies
\begin{equation}\label{K1}
\frac{1}{c_k}\leq \frac{\epsilon}{G \sqrt{\sigma \gamma c_0 \alpha_0}}.
\end{equation}
where \(G\) is the Lipschitz constant of \(f\) and \(\alpha_{i,k}\) is bounded by \(\frac{c_0 \alpha_0}{c_k}\).
Under Assumption 1, there exists a constant  $K'_3$ for \(k > K'_3\), by the consensus result, we get
\begin{equation}\label{K3}
\frac{1}{n} \sum_{i=1}^{n} \|{\bar{x}_k - z_{i,k}}\| \leq \frac{\epsilon}{4G}.
\end{equation}
Combing \eqref{K1} and \eqref{K3}, when \(k > \max\{K_{1}^{\prime}, K_{3}^{\prime}\}\), 
\begin{equation}\label{x bound3}
	\begin{aligned}[b]
		& \frac{1}{n}\sum_{i=1}^{n} \|x_{i,k+1} - x^{*}\|^{2} - \frac 1n \sum_{i=1}^n \| x_{i,k} - x^{\star} \|^2 \\ 
		& \leq G\alpha_{\text{min}}(2 -\frac{\gamma}{c_k})\frac 1n \sum_{i=1}^n  \|z_{i,k}-\bar{x}_k\| \\
		& \quad - G\alpha_{\text{min}}(2 -\frac{\gamma}{c_k})(f(\bar{x}_k)-f(x^{\star})) \\
		& \quad + \frac{\gamma c_0 \alpha_0}{c_k^2} \frac 1n \sum_{i=1}^n   (f_i(x^{\star})-\bar{f}_i^k) \\
		& \leq \alpha_{\text{min}}(2 -\frac{\gamma}{c_k}) \frac{\epsilon}{4}- \alpha_{\text{min}}(2 -\frac{\gamma}{c_k}) \epsilon \\
		& \quad + \frac{\gamma c_0 \alpha_0}{c_k^2} \frac 1n \sum_{i=1}^n   (f_i(x^{\star})-\bar{f}_i^k)
	\end{aligned}
\end{equation}
building upon Lemma \ref{psvd}, we establish 
$
\frac{1}{n}\sum_{i=1}^{n}(f_i(x^{\star})-\bar{f}_i^k) \leq f(x^{\star})-\bar{f}^0=\sigma,
$
%There exists a sufficiently large integer $K_5^{'} \in \mathbb{N}^+$ such that for all $k > K_5^{'}$, the avrage optimality gap satisfies:
%\[
%\frac{1}{n}\sum_{i=1}^{n}(f_i(x^{\star})-\bar{f}_i^k)-\sigma \leq \epsilon
%\]
%here, $K_0 = \max\{K_1^{'}, K_2^{'}, K_3^{'}\}$ encapsulates the conditions required for this bound to hold. 
under this condition, the evolution of the average squared distance of the iterates $x_{i,k}$ to an optimal solution $x^{\star}$ can be bounded, we have
\begin{align}\label{epsilon}
	&\frac{1}{n}\sum_{i=1}^{n}\|x_{i,k+1} - x^{*}\|^{2} - \frac{1}{n}\sum_{i=1}^{n}\|x_{i,k}- x^{*}\|^{2}\notag\\
	& \leq \sigma \left( \frac{\gamma c_0 \alpha_0}{c_k^2} - G\alpha_{\text{min}}\left(2 -\frac{\gamma}{c_k}\right) \frac{3\epsilon}{4\sigma} \right)\notag,
\end{align}
combining \eqref{x bound3}, there exist $\max\{K_{1}^{\prime}, K_{3}^{\prime}, K_\epsilon\} \in \mathbb{N}^+$, such that for all $k > \max\{K_{1}^{\prime}, K_{3}^{\prime}, K_\epsilon\}$,
\begin{align*}
\frac 1n \sum_{i=1}^n \| x_{i,k+1} - x^{\star} \|^2  
	\leq  \frac 1n \sum_{i=1}^n \| x_{i,k} - x^{\star} \|^2  - \mathcal{O}(\frac{1}{c_k})\epsilon,
\end{align*}	
given that $c_k = \sqrt{k+1}$, this implies that there exists a constant $C > 0$ such that for all sufficiently large $k$:
\begin{equation*} 
	\frac{1}{n}\sum_{i=1}^{n} \|x_{i,k+1} - x^{\star}\|^2 \leq \frac{1}{n}\sum_{i=1}^{n} \|x_{i,k} - x^{\star}\|^2 - \frac{C \epsilon}{\sqrt{k+1}}.
\end{equation*}
Summing this inequality from $k=K_0$ to $K_0+m-1$ yields
\begin{align}\label{eq:main_inequality}
&\frac{1}{n}\sum_{i=1}^{n}  \|x_{i,K_0+m} - x^{\star}\|^2 - \frac{1}{n}\sum_{i=1}^{n} \|x_{i,K_0} - x^{\star}\|^2 \notag
\\
&\leq - C \epsilon \sum_{k=K_0}^{K_0+m-1} \frac{1}{\sqrt{k+1}},
\end{align}	
the series $\sum_{k=1}^\infty \frac{1}{\sqrt{k+1}}$ is a $p$-series with $p=1/2$, which is known to be divergent. Therefore, as $m \to \infty$, the partial sum also diverges to infinity
\[
\lim_{m \to \infty} \sum_{k=K_0}^{K_0+m-1} \frac{1}{\sqrt{k+1}} = \infty,
\]
this means we can choose a sufficiently large $m$ such that the negative term on the right-hand side of \eqref{eq:main_inequality} dominates the constant positive term
\[
C \delta \sum_{k=K_0}^{K_0+m-1} \frac{1}{\sqrt{k+1}} > \frac{1}{n}\sum_{i=1}^{n} \|x_{i,K_0} - x^{\star}\|^2,
\]
for such a large $m$, it follows that
\[
\frac{1}{n}\sum_{i=1}^{n} \|x_{i,K_0+m} - x^{\star}\|^2 < 0.
\]

This is a contradiction, since the term $\frac{1}{n}\sum_{i=1}^{n} \|x_{i,K_0+m} - x^{\star}\|^2$ is an average of squared norms and must be non-negative. So $\liminf_{k \to \infty} f(\bar{x}_k) - f(x^{\star}) \leq 0$. As a result, $\liminf_{k \to \infty} f(\bar{x}_k) - f(x^{\star}) = 0$. So there exists a subsequence $\{y(k_\nu)\}$ of $\{\bar{x}_k\}$ such that $\lim_{k_\nu \to \infty} f(y(k_\nu)) = f(x^{\star})$.
\end{IEEEproof}

\begin{IEEEbiography}
	[{\includegraphics[width=1.1in,height=1.25in,clip,keepaspectratio]{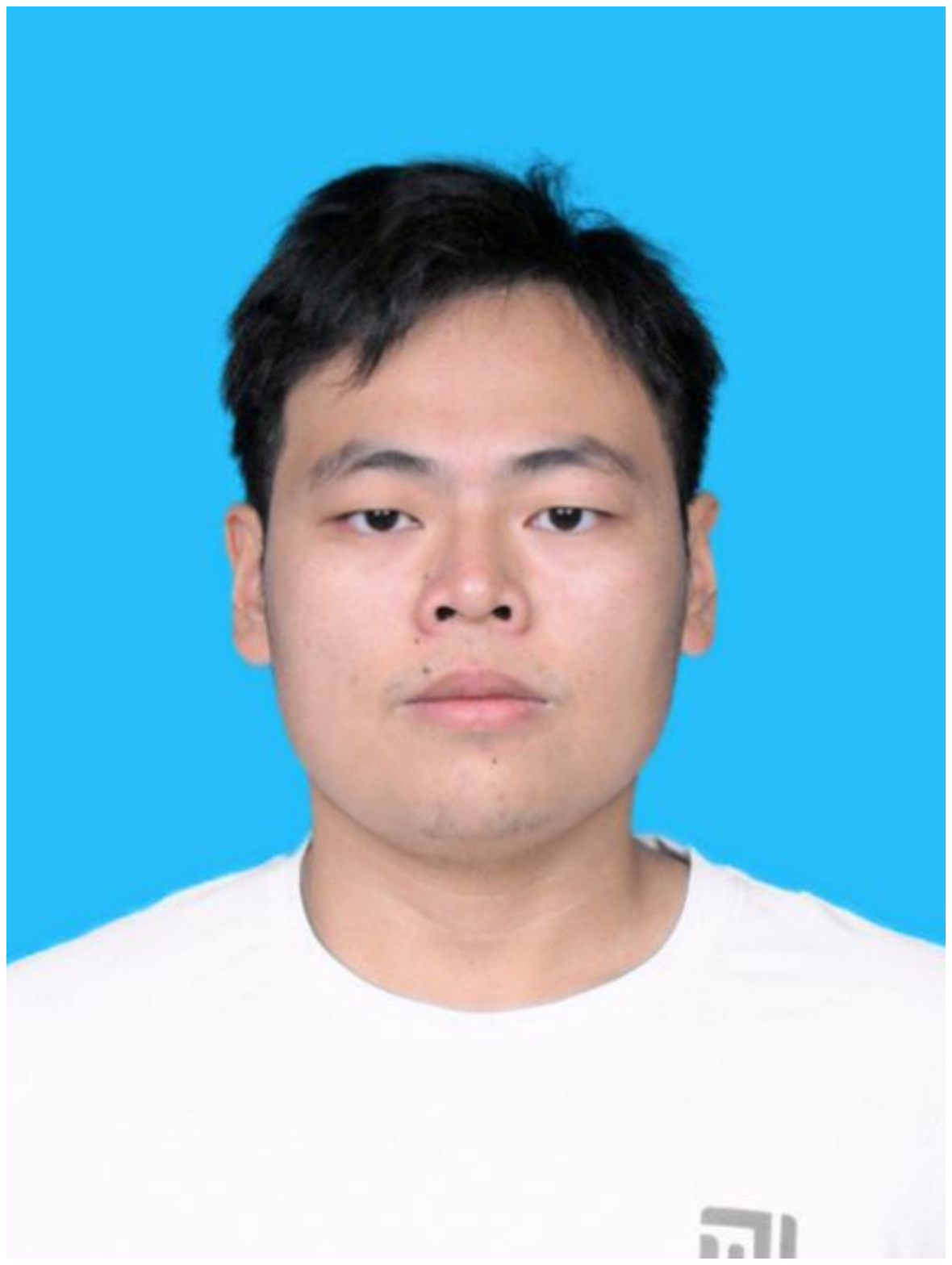}}]{Chen Ouyang}
	received the B.S. degree in information and computational science, the M.E. degrees in School of Mathematics and Information Science from Guangxi University, Nanning, China, in 2022, 2025, respectively. He is now pursuing his Ph.D. in the School of Intelligent Systems Engineering, Sun Yat-sen University. His current research interests include distributed optimization and machine learning.
\end{IEEEbiography}
\vspace{-8cm}
\begin{IEEEbiography}
	[{\includegraphics[width=1in,height=1.25in,clip,keepaspectratio]{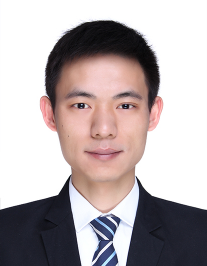}}]{Yongyang Xiong}
	received the B.S. degree in information and computational science, the M.E. and Ph.D. degrees in control science and engineering from Harbin Institute of Technology, Harbin, China, in 2012, 2014, and 2020, respectively. From 2017 to 2018, he was a Joint Ph.D.Student with the School of Electrical and Electronic Engineering, Nanyang Technological University, Singapore. From 2021 to 2024, he was a Postdoctoral Fellow with the Department of Automation, Tsinghua University, Beijing, China. Currently, he is an associate professor with the School of Intelligent Systems Engineering, Sun Yat-sen University. His current research interests include networked control systems, distributed optimization and learning, multi-agent reinforcement learning and their applications.
\end{IEEEbiography}
\vspace{-8cm}
\begin{IEEEbiography}
	[{\includegraphics[width=1in,height=1.25in,clip,keepaspectratio]{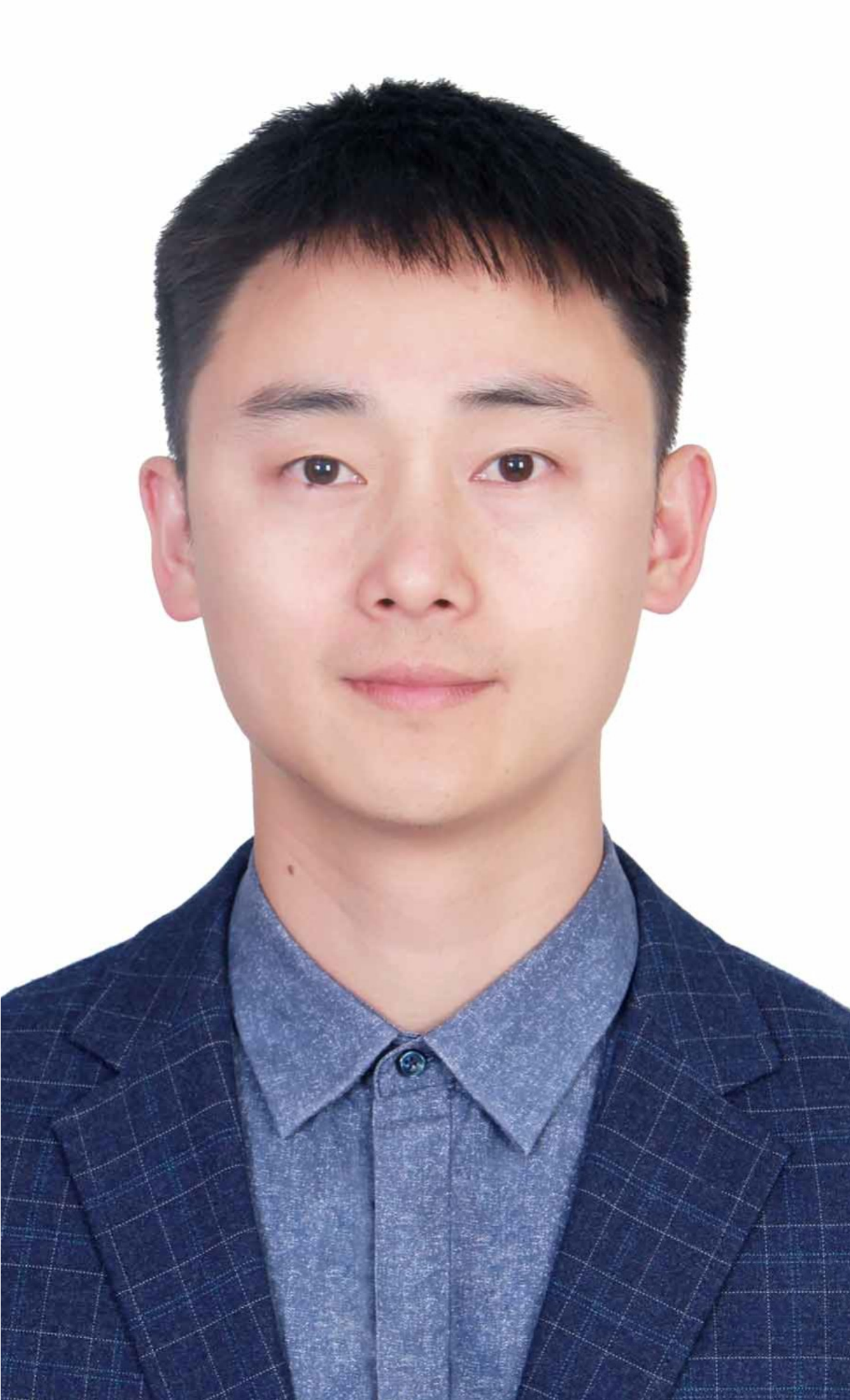}}]{Jinming Xu}
	received the B.S. degree in Mechanical Engineering from Shandong University,
	China, in 2009 and the Ph.D. degree in Electrical and Electronic Engineering from Nanyang
	Technological University (NTU), Singapore, in
	2016. From 2016 to 2017, he was a Research
	Fellow at the EXQUITUS center, NTU; he was
	also a postdoctoral researcher in the Ira A.
	Fulton Schools of Engineering, Arizona State
	University, from 2017 to 2018, and the School of
	Industrial Engineering, Purdue University, from
	2018 to 2019, respectively. In 2019, he joined Zhejiang University,
	China, where he is currently a Professor with the College of Control
	Science and Engineering. His research interests include distributed
	optimization and control, machine learning and network science. He
	has published over 50 peer-reviewed papers in prestigious journals
	and leading conferences. He has been the Associate Editor of IEEE
	TRANSACTIONS on SIGNAL and INFORMATION PROCESSING OVER NETWORKS.
\end{IEEEbiography}
\newpage
\begin{IEEEbiography}
	[{\includegraphics[width=1in,height=1.25in,clip,keepaspectratio]{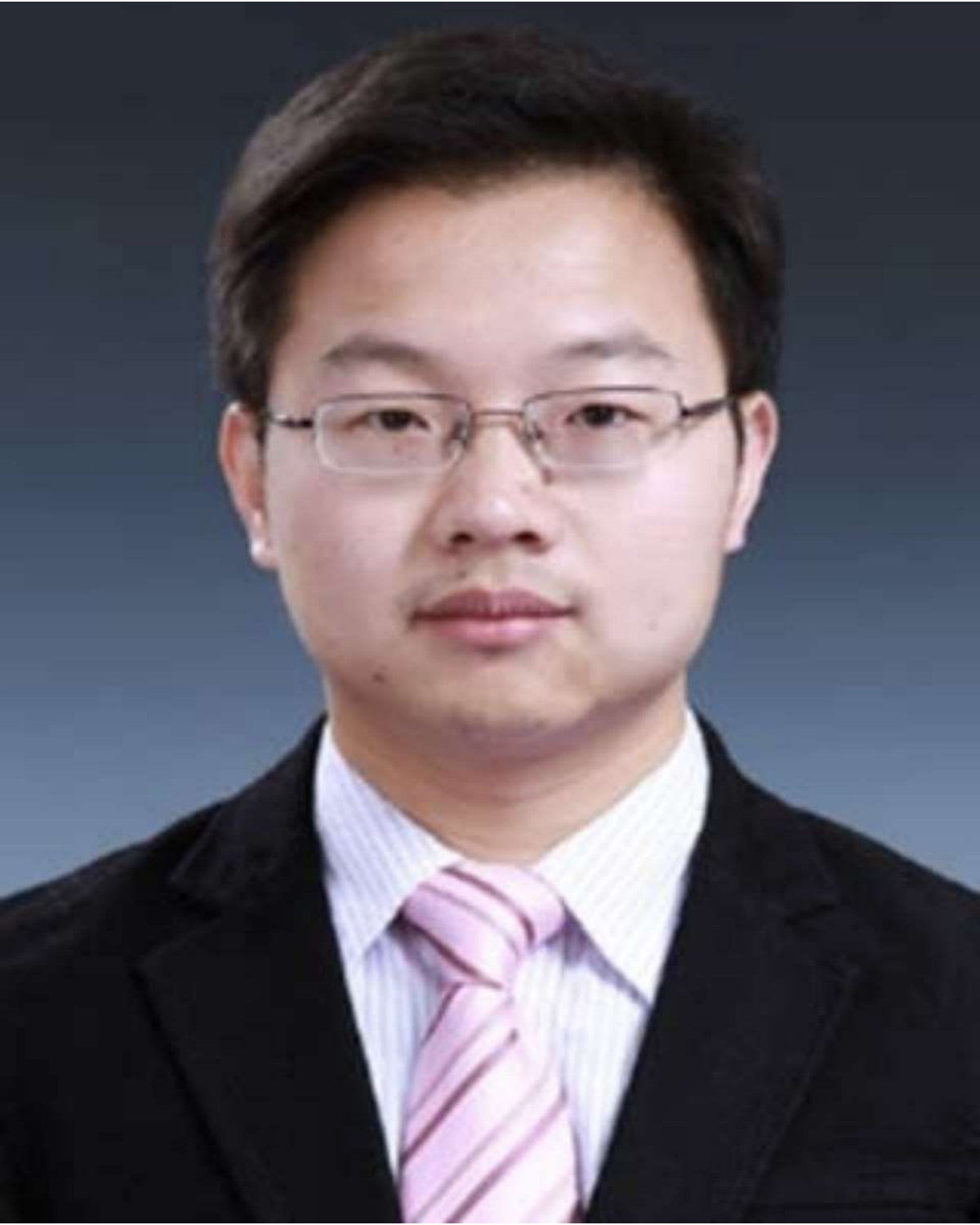}}]{Keyou You}
(Senior Member, IEEE) received the B.S. degree in statistical science from Sun Yat-sen University, Guangzhou, China, in 2007 and the Ph.D. degree in electrical and electronic engineering from Nanyang Technological University (NTU), Singapore, in 2012. 

After briefly working as a Research Fellow at NTU, he joined Tsinghua University, Beijing, China, where he is currently a Full Professor in the Department of Automation. He held visiting positions with Politecnico di Torino, Turin, Italy, Hong Kong University of Science and Technology, Hong Kong, China, University of Melbourne, Melbourne, Victoria, Australia, and so on. His research interests include the intersections between control, optimization and learning, as well as their applications in autonomous systems. 

Dr. You received the Guan Zhaozhi Award at the 29th Chinese Control Conference in 2010 and the ACA (Asian Control Association) Temasek Young Educator Award in 2019. He received the National Science Funds for Excellent Young Scholars in 2017 and for Distinguished Young Scholars in 2023. He is currently an Associate Editor for \textit{Automatica} and IEEE TRANSACTIONS ON CONTROL OF NETWORK SYSTEMS.
\end{IEEEbiography}
\vspace{-4cm}
\begin{IEEEbiography}
	[{\includegraphics[width=1in,height=1.25in,clip,keepaspectratio]{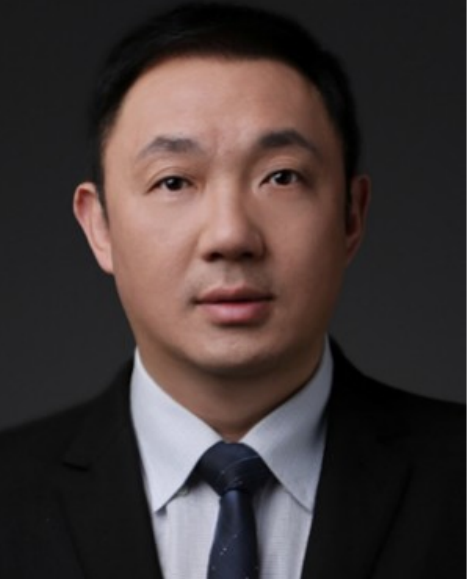}}]{Yang Shi} (Fellow, IEEE) 
received the B.Sc. and Ph.D. degrees in mechanical engineering and
automatic control from Northwestern Polytechnical University, Xi’an, China, in 1994 and 1998, respectively, and the Ph.D. degree in electrical and computer engineering from the University of Alberta, Edmonton, AB, Canada, in 2005. He was a Research Associate with the Department of Automation, Tsinghua University, China, from 1998 to 2000. From 2005 to 2009, he was an Assistant Professor and an Associate Professor with the Department of Mechanical Engineering, University of Saskatchewan, Saskatoon, SK, Canada. In 2009, he joined the University of Victoria, and currentlu he is a Professor with the Department of Mechanical Engineering, University of Victoria, Victoria, BC, Canada. His current research interests include networked and distributed systems, model predictive control (MPC), cyber-physical systems (CPS), robotics and mechatronics, navigation and control of autonomous systems (AUV and UAV), and energy system applications.

Dr. Shi is the IFAC Council Member. He is a fellow of ASME, CSME,
Engineering Institute of Canada (EIC), Canadian Academy of Engineering (CAE), Royal Society of Canada (RSC), and a registered Professional Engineer in British Columbia and Canada. He received the University of Saskatchewan Student Union Teaching Excellence Award in 2007, the Faculty of Engineering Teaching Excellence Award in 2012 at the University of Victoria (UVic), and the 2023 REACH Award for Excellence in Graduate Student Supervision and Mentorship. On research, he was a recipient of the JSPS Invitation Fellowship (short-term) in 2013, the UVic Craigdarroch Silver Medal for Excellence in Research in 2015, the Humboldt Research Fellowship for Experienced Researchers in 2018, CSME Mechatronics Medal in 2023, the IEEE Dr.-Ing. Eugene Mittelmann Achievement Award in 2023,
the 2024 IEEE Canada Outstanding Engineer Award. He was a Vice-President on Conference Activities of IEEE IES from 2022 to 2025 and the Chair of IEEE IES Technical Committee on Industrial Cyber-Physical Systems. Currently, he is the Editor-in-Chief of IEEE TRANSACTIONS ON INDUSTRIAL ELECTRONICS. He also serves as Associate Editor for \textit{Automatica}, IEEE TRANSACTIONS ON AUTOMATIC CONTROL, and \textit{Annual Review in Controls}.
\end{IEEEbiography}
% TODO: \usepackage{graphicx} required

\end{document}